\documentclass[11pt]{article}
= 0.0in \headheight = 0.0in \topmargin = 0.3in
\def\Dj{\hbox{D\kern-.73em\raise.30ex\hbox{-}
\raise-.30ex\hbox{}}}
\def\dj{\hbox{d\kern-.33em\raise.80ex\hbox{-}
\raise-.80ex\hbox{\kern-.40em}}}
\usepackage{epsfig}
\usepackage{amsmath,amsthm,amsfonts,amssymb,amscd,cite}
\allowdisplaybreaks

\usepackage{color}
\usepackage{float}
\usepackage{graphicx}
\usepackage{subfigure}
\usepackage{mathtools}
\usepackage{blkarray}
\usepackage{xcolor}
\usepackage{multirow}
\usepackage{rotating}
\usepackage{hhline}
\usepackage{array}
\usepackage{tikz}

\usepackage[ruled,linesnumbered]{algorithm2e}
\usepackage{caption}
\usepackage{tgtermes}
\usepackage[T1]{fontenc}
\usepackage{tipa}
\usepackage{multirow}
\usepackage{graphicx}
\usepackage{booktabs}
\usepackage{multirow}
\usepackage{mathrsfs}
\allowdisplaybreaks
\usepackage{mathrsfs}

\usepackage{color}
\usepackage{float}
\usepackage{graphicx}
\usepackage{mathtools}
\usepackage{blkarray}
\usepackage{xcolor}
\usepackage{multirow}
\usepackage{rotating}
\usepackage{hhline}
\usepackage{array}
\usepackage{tikz}
\usepackage{array}\newcolumntype{C}[1]{>{\centering\arraybackslash}p{#1}}
\newtheorem{theo}{Theorem}[section]

\newtheorem{lemma}[theo]{Lemma}

\newtheorem{coro}[theo]{Corollary}

\newtheorem{problem}[theo]{Problem}

\newtheorem{rem}{Remark}
\newtheorem{defi}[theo]{Definition}

\newcolumntype{C}[1]{>{\centering\arraybackslash}p{#1}}

\begin{document}

\pagestyle{plain}

\title{Extremal graphs  for the maximum $A_{\alpha}$-spectral radius of graphs with order and size }

\author{Jie Zhang$^a$\thanks{ J. Zhang is partly supported by National Natural Science Foundation of China (No. 12371349)}, Ya-Lei Jin$^b$\thanks{Y.-L. Jin partly supported by National Natural Science Foundation of China (Nos. 12371349 and 12471331)}, Hua Wang$^c$, Jin-Xuan Yang$^d$\thanks{J.-X Yang is partly supported by National Natural Science Foundation of China (No. 12361073)}, Xiao-Dong Zhang$^e$\thanks{Corresponding author. E-mail addresses:
xiaodong@sjtu.edu.cn.  X.-D. Zhang is partly supported by National Natural Science Foundation of China (Nos. 12371354), the Science and Technology Commission of Shanghai Municipality (No. 22JC1403600) and  the Montenegrin-Chinese Science and Technology
Cooperation Project (No. 4-3).}
\\{\small \it  a.  School of Insurance, Shanghai Lixin University of Accounting and Finance, }\\{\small \it Shanghai, 201209, P. R. China.}
\\{\small \it  b. Department of Mathematics, Shanghai Normal University, }\\{\small \it Shanghai 200234, P. R. China}
\\{\small \it  c. Department of Mathematical Sciences,  Georgia Southern University, }\\{\small \it Statesboro, GA 30460, USA}
\\{\small \it  d. School of Statistics and Mathematics, Yunnan University of Finance and Economics, }\\{\small \it Kunming, 650221, P.R. China,}
\\{\small \it  e. School of Mathematical Sciences, MOE-LSC and SHL-MAC, }\\{\small \it Shanghai Jiao Tong University, }\\{\small \it
800 Dongchuan Road, Shanghai 200240,  P.R. China}\\
%Email: zhangjie.sjtu@163.com, yaleijin@shnu.edu.cn,   hwang@georgiasouthern.edu,yangjinxuan_2007@163.com,  xiaodong@sjtu.edu.cn
}

\date{}
\maketitle

\begin{abstract}
 In 1986, Brualdi and Solheid  firstly proposed the problem of determining the  maximum spectral radius of graphs in   the set $\mathcal{H}_{n,m}$ consisting of all simple connected graphs with $n$ vertices and $m$ edges, which  is a very tough problem and  far from resolved. The $A_{\alpha}$-spectral radius of  a simple graph  of order $n$, denoted by $\rho_\alpha(G)$, is the largest eigenvalue of the matrix
   $A_{\alpha}(G)$ which is defined as  $\alpha D(G)+(1-\alpha)A(G)$ for $0\le \alpha< 1$, where  $D(G)$ and $A(G)$ are the degree diagonal and adjacency  matrices of $G$, respectively.  In this paper, if  $r$ is a positive integer, $n>30r$ and $n-1\leq m \le rn-\frac{r(r+1)}{2}$, we characterize all extremal graphs  which have the maximum  $A_{\alpha}$-spectral radius of graphs in   the set $\mathcal{H}_{n,m}$. Moreover, the problem  on $A_{\alpha}$-spectral radius proposed by
   Chang and Tam [T.-C. Chang and B.-T. Tam,  Graphs of fixed order and size with maximal $A_{\alpha}$-index. Linear Algebra Appl. 673 (2023), 69-100] has been solved.

\end{abstract}
\vspace{3mm}

 \smallskip


\maketitle

\section{Introduction}
\label{sec:intro}

In this paper, we only consider  finite undirected  graphs without multiple edges or loops. Let $G$ be a simple graph with vertex set $V(G)=\{v_1,v_2,\ldots,v_n\}$ and edge set $E(G)$. Denote by $|V(G)|=n$ and $|E(G)|=m$ the number of vertices and edges of G, respectively. Let $N_G(u)$ be the set of vertices which are adjacent to $u$ in $G$,
if there is no ambiguous, we denote by $N(u)$ for simplify.  Moreover, denote by $d_i=d_G(v_i)$ the degree of vertex $v_i$ for $i=1, 2, \ldots, n$.
 For two distinct vertices $u, v\in V(G)$, if $u$ is adjacent to $v$ then denote  $u\sim v$,  or for short $uv$;
if $u$ is not adjacent to $v$ then denote  $u\not\sim v$. For $uv\not\in E(G)$, let $G+uv$ be the graph which is obtained from $G$ by adding the new edge $uv$.
For $uv\in E(G)$, let $G-uv$ be the graph which is obtained from $G$ by deleting the edge $uv$.

  The adjacency matrix of  a simple graph $G$ is the $n \times n$ matrix $(a_{ij})_{n\times n}$, where $a_{ij}=1$ if
$v_i$ is adjacent to $v_j$, and 0 otherwise. Moreover, the largest eigenvalue of $A(G)$ is usually called the
spectral radius of $G$.   For given two positive integers $n$ and $m\ge n-1$, let $\mathcal{H}_{n,m}$ be the set of all connected graphs of order $n$ with size $m$,  and let  $\mathcal{G}_{n,m}$ be the set of all  graphs of order $n$ with size $m$. %i.e., $\mathcal{H}_{n,m} =\{G:  G \mbox{ is connected  with}\ |V(G)|=n, |E(G)|=m\}$.
  In 1986, Brualdi and Solheid \cite{bruadis} firstly proposed the following  problem:

%Brualdi and Solheid \cite{bruadis} considered the following problem:

\begin{problem}\cite{bruadis}\label{que1}
 For given two positive integers $n$ and $m\ge n-1$, determine  the maximum spectral radius  of graphs in  $\mathcal{H}_{n,m} $ and characterize all extremal graphs which attain the maximum value.
\end{problem}

Further, for $n\le m\le n+5$,  Brualdi and Solheid \cite{bruadis}  characterized all graphs which have  the maximum spectral radius in  $\mathcal{H}_{n,m}$.  Later, in 1988, for a given positive integer $r\ge 3$ and $m=n+r$, Cvetkovi\'c and Rowlinson \cite{cvet} proved that $S_{n, m}$ is the unique graph with maximum spectral radius in $\mathcal{H}_{n,m} $  for sufficiently large $n$ (the definition of $S_{n,m}$ is given in Section 2).  In 1991, for  $m=n+\binom{r}{2}-1$  with  positive integer $r$,   Bell \cite{bell1} determined all graphs which have maximum spectral radius in  $\mathcal{H}_{n,m} $.    However, up to now,  for given any two integers $n$ and $m\ge n-1$, Problem \ref{que1}  is far from being completely resolved and seems to be very tough.

 In 2017, Nikiforov\cite{niki} introduced the $A_{\alpha}$-matrix of  a simple graph $G$  which is defined to be
$A_{\alpha}(G)=\alpha D(G)+(1-\alpha)A(G)$,
where $\alpha \in [0,1)$. The largest eigenvalues of $A_{\alpha}(G)$ is denoted by $\rho_{\alpha}(G)$, which is called the $A_{\alpha}$-spectral radius of $G$. If $G$ is connected, there exists a unique positive eigenvector corresponding to $\rho_{\alpha}(G)$, which is called the Perron vector of $A_{\alpha}$.
Notice that $A_0(G)=A(G)$ and $A_{\frac{1}{2}}(G)=\frac{1}{2}Q(G)$, where $Q(G)$ is the signless Laplacian matrix of $G$. Hence  Problem \ref{que1} may be generalized to the following problem.

\begin{problem}\label{prob2}
Determine all graphs which have the maximum $A_{\alpha}$-spectral radius in $\mathcal{H}_{n, m}$.
\end{problem}

  Nikiforov, Past\'{e}n, et al. \cite{nikig} proved that the star graph $S_n$ is the unique graph maximizing the $A_{\alpha}$-spectral radius in $\mathcal{H}_{n,n-1}$. Recently, Li, Tam, et.al. in \cite{lity} have made  significant progress on Problem \ref{prob2}.

\begin{theo}\label{t2n-3c}\cite{lity}
Let $n$ and $m$ be two positive integers with $n-1 \le m \le 2n-3$. \\
 {\em (1).} If $\alpha \in (\frac{1}{2},1)$ or $\alpha=\frac{1}{2}$ and $m \neq n+2$, then $S_{n,m}$ is the unique graph that maximizes the $A_{\alpha}$-spectral radius in $\mathcal{H}_{n,m}$.\\
{\em (2).} If $\alpha=\frac{1}{2}$ and $m = n+2$, then $S_{n,n+2}$ and $L_{n,n+2}$ are the two precisely connected graphs that maximize the $A_{\frac{1}{2}}$-spectral radius in $\mathcal{H}_{n,n+2}$ (the definition of  $L_{n,m}$ is given in Section 2).
\end{theo}

   On the other hand,   the problem of determining the maximal $A_{\alpha}$-spectral radius in $ \mathcal{G}_{n, m}$ has attracted much attention.   When $\alpha=0$ (respectively $\alpha=1/2$), the maximal $A_{\alpha}$-spectral radius problem becomes the well-known maximal spectral radius problem (respectively, the maximal $Q$-spectral radius problem).   Whereas the maximal spectral radius  problem over the class $\mathcal {G}_{n,m}$ has been studied by Brualdi and Hoffman \cite{bruadih}, Friedland \cite{fried} and has been completely solved \cite{rowl}. The maximal $Q$-spectral radius in $\mathcal {G}_{n,m}$ has also been investigated by Chang and Tam \cite{changtam}, An\textcrd eli\'{c}  et. al. \cite{ak}, etc. The maximal $A_{\alpha}$-spectral radius problem over the class $\mathcal {G}_{n,m}$ has been treated by  Chang and Tam \cite{changtam2026}, Chen and  Huang \cite{chen}, Li and Qin \cite{liq}, etc.  Recently,  Li and Tam et al \cite{lity} proposed the following problem for $m=2n-2$.
\begin{problem}\cite{changtam} \label{que}
  For given two positive integers $n\ge 4$ and $m=2n-2$, characterize graphs that maximize $A_{1/2}$-spectral radius  of  graphs in $\mathcal{G}_{n,m}$.
\end{problem}

Motivated by the above problems, we have investigated the $A_\alpha$-spectral radius of a simple connected graph of order $n$  with size $m$. The main results can be stated as follows.

\begin{theo}\label{t2n-2-1}
Let $n\ge 4$ and $m = 2n-2$. If  $G'$  is any graph having the maximum  $A_{1/2}$-spectral radius of  graphs in $\mathcal{G}_{n,m}$  which consists of all graphs of order $n$ with size $m$, then $G'\cong K_5 \cup K_1$ for $n=6$, and $G'\cong S_{n,2n-2}$ for $n \neq 6$.
\end{theo}

Theorem~\ref{t2n-2-1} fully resolves the  problem  proposed by  Chang and Tam \cite{changtam}.

\begin{theo}\label{tknc}
Let $r$, $n$ and $m$ be three positive integers satisfying $r\ge 3$,  $n >\frac{30r-63+5\sqrt{32r^2-136r+137}}{2}$  and $n-1\le m\le rn-\frac{r(r+1)}{2}$. \\
{\em (1).}  If $\alpha \in (\frac{1}{2}, 1)$,  or $\alpha=\frac{1}{2}$ and $m \neq (r-1)n-\frac{r(r-1)}{2}+3$, then $S_{n, m}$ is the only extremal graph that maximizes the $A_{\alpha}$-spectral radius in $\mathcal{H}_{n,m}$. \\
{\em (2).}  If $\alpha=\frac{1}{2}$ and $m = (r-1)n-\frac{r(r-1)}{2}+3$, then $S_{n, m}$ and $\tilde{S}_{n, m}$ are the  two extremal graphs that  maximize the $A_{\frac{1}{2}}$-spectral radius in $\mathcal{H}_{n,m}$ (the definition of $\tilde{S}_{n, m}$ is given in Section 2).
\end{theo}

The rest of this paper is organized as follows.  In Section 2, we introduce some definitions and some known results  which  are useful in this paper. In Section 3, several new graph transformations of $A_{\alpha}$-spectral radius are proposed  which will be interesting on their own.  In Section 4, based on these  transformations,  the proofs of  Theorems \ref{t2n-2-1} and \ref{tknc} are presented.

\section{Preliminaries }

In this section, we introduce some notations and some known results which will be used later.
Let $H_1$ and $H_2$ be two disjoint graphs. Denote by $H_1\bigcup H_2$ the sum of $H_1$ and $H_2$, where $V(H_1\bigcup H_2)=V(H_1)\bigcup V(H_2), E(H_1\bigcup H_2)=E(H_1)\bigcup E(H_2)$.
Denote by $H_1\bigvee H_2$  the product of $H_1$ and $H_2$, obtained by adding all edges between $H_1$ and $H_2$, i.e.
$V(H_1\bigvee H_2)=V(H_1)\bigcup V(H_2)$, the edges set of $H_1\bigvee H_2$ consisting of $E(H_1)\bigcup E(H_2)$ and $\{uv\}$ for each $u\in V(H_1)$ and $v\in V(H_2)$.
Moreover, denote by $K_n$ the complete graph of order $n$,  $K_{1,n-1}$ the star of order $n$, $\bar{K}_{n}$(or $nK_1$) the graph consisting $n$ isolated vertices. In particular, $K_{1,0}$ has only one isolated vertex.

 For two positive integers $n$, $m$ with  $n-1\le m$,  let  $k$ be the largest integer such that $m\geq \sum_{i=1}^{k}(n-i)$ and ${a}=m-\sum_{i=1}^{k}(n-i)$.  The graph  $S_{n,m}$ of order $n$ with size $m$  is defined to be
   $$S_{n,m}=K_{k}\bigvee \left(K_{1, a}\bigcup (n-a-k-1)K_1\right),$$
    which is called quasi-star graph.
 Clearly, if $m=n-1$, then $S_{n,m}$ is the star $K_{1, n-1}$;  if $m=n(n-1)/2$, then $S_{n,m}$ is the complete graph $K_n$.
 In  addition, if  $m = kn-\frac{k(k+1)}{2}+3$, then $a=3$ and   the graph $\tilde{S}_{n, m}$ of order $n$ with size $m$ is defined to be
  $$\tilde{S}_{n, m}=K_{k}\bigvee \left(K_3\bigcup (n-k-3)K_1\right).$$
On the other hand, let  $\bar{k}$  be the largest integer such that $m-n+1\geq \sum_{i=1}^{\bar{k}-1}i$ and $\bar{a}=m-n+1-\sum_{i=1}^{\bar{k}-1}i$.
The graph $L_{n,m}$  of order $n$ with size $m$ is defined to be
$$L_{n,m}=\left \{ \begin{array}{ll}
\left(K_{\bar{k}}\bigcup (n-\bar{k}-1)K_1\right)\bigvee K_1, & {  \mbox{for  } \bar{a}=0};\\
\left(K_{\bar{a}}\bigvee \left( K_{\bar{k}-\bar{a}}\bigcup K_{1}\right)\bigcup (n-\bar{k}-2)K_1\right)\bigvee K_1, & {   \mbox{for } \bar{a}> 0  }.
 \end{array}\right.$$
%$L_{n,m}=\left(K_{k}\cup (n-k-1)K_1\right)\vee K_1$  for $a=0$,
% and\\ $L_{n,m}=\left(K_{a}\vee \left( K_{k-a-1}\cup K_{1}\right)\cup (n-k-1)K_1\right)\vee K_1$ for $a> 0$.
%, where $k$ is the largest integer such that
%$m\geq \sum_{i=1}^{k}(n-i)$ and $a=m-\sum_{i=1}^{k}(n-i)$.
 For example, if $n=6$ and $m=10$, then $S_{6,10}$ and $L_{6,10}$ are depicted in Figure \ref{fig:11}, where $k=2$, ${a}=1$;  $\bar{k}=3$, $\bar{a}=2$.
\begin{figure}[htbp]
\centering  %
\subfigure[ $S_{6,10}$]{
\begin{minipage}{3cm}
\centering
\begin{tikzpicture}[scale=.8]
        \node[fill=white,circle,inner sep=1pt] (t1) at (-2,0) {1};
        \node[fill=white,circle,inner sep=1pt] (t2) at (2,0) {2};
		\node[fill=white,circle,inner sep=1pt] (t3) at (-1,1) {3};
        \node[fill=white,circle,inner sep=1pt] (t4) at (1,1) {4};
        \node[fill=white,circle,inner sep=1pt] (t5) at (-1,-1) {5};
        \node[fill=white,circle,inner sep=1pt] (t6) at (1,-1) {6};

        \draw (t1)--(t2) (t1)--(t3) (t1)--(t4) (t1)--(t5) (t1)--(t6);
        \draw  (t2)--(t6) (t2)--(t4) (t2)--(t3) (t2)--(t5) (t3)--(t4);
        \end{tikzpicture}
%\caption{Graph S(6,9)}
\end{minipage}
}
\subfigure[ $L_{6,10}$]{ %
\begin{minipage}{5cm}
\centering    %
\begin{tikzpicture}[scale=.8]
        \node[fill=white,circle,inner sep=1pt] (t1) at (-2,0) {1};
        \node[fill=white,circle,inner sep=1pt] (t2) at (2,0) {2};
		\node[fill=white,circle,inner sep=1pt] (t3) at (-1,1) {3};
        \node[fill=white,circle,inner sep=1pt] (t4) at (1,1) {4};
        \node[fill=white,circle,inner sep=1pt] (t5) at (-1,-1) {5};
        \node[fill=white,circle,inner sep=1pt] (t6) at (1,-1) {6};

        \draw (t1)--(t2) (t1)--(t3) (t1)--(t4) (t1)--(t5) (t1)--(t6);
        \draw  (t2)--(t5) (t2)--(t4) (t2)--(t3) (t3)--(t4) (t3)--(t5);
        \end{tikzpicture}
%\caption{Graph S(6,9)}
\end{minipage}
}
\caption{ $S_{n,m}$ and $L_{n,m}$ with $n=6$, $m=10$}
\label{fig:11}
\end{figure}
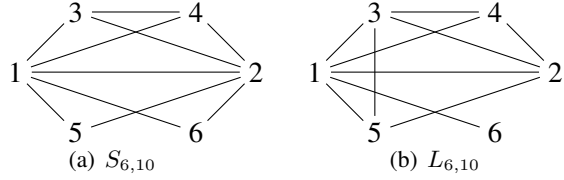

%%定义$\tilde{S}_{n, m}$ 并且画图

A simple graph $G = (V, E)$ is called a threshold graph if $G$ has no induced subgraph of the forms $2K_2$, $C_4$, $P_4$. Clearly, $S_{n,m}$ and $L_{n,m}$ are threshold graphs.
Threshold graphs have a beautiful structure and possess many important mathematical properties such as being the extreme cases of certain graph
properties (see \cite{peled1999}). For more information on threshold graphs, the reader may refer to the monograph \cite{mp}. We  state the following characterizations of threshold graphs which are useful in this paper.
\begin{lemma} \label{1} \cite{peled1999}  Let $G=(V, E)$ be a simple graph with degree sequence $(d_1, d_2,\ldots, d_n)$.
The following statements are equivalent:
\\ {\em (1).}  $G$ is a threshold graph.
\\ {\em (2).} $G$ can be constructed from the one-vertex graph by repeatedly adding an isolated vertex or a universal ( or dominating ) vertex ( a vertex adjacent to every other vertex ).
\\ {\em (3).} Every three distinct vertices $i, j, k$ of $G$ satisfy:
if $d_i \ge d_j$ and $jk$ is an edge, then $ik$ is an edge.
\end{lemma}
\begin{lemma} \label{54} \cite{hk}
Any threshold graph  is uniquely determined by  its degree sequence.
\end{lemma}
\noindent %\vspace{2mm}

%\begin{lemma}\label{nkc} \cite{lity}
%If $\alpha \in(1/2,1)$ and $3\le k \le n-3$, then $S_{n,n+k}$ is the unique graph with the maximum $A_{\alpha}$-index over
%all connected graphs in $\mathcal{H}_{n,n+k}$.
%\end{lemma}

%\begin{lemma}\label{nk} \cite{zhaih}
%If $k\ge 3$ and $n \ge k+3$, then $S_{n,n+k}$ is the unique graph with the maximum Q-index over
%all graphs in $\mathcal{G}_{n,n+k}$.
%\end{lemma}

Let $D=(d_1, d_2,\ldots, d_n)$ be a non-increasing positive integral sequence. The
 Ferrers matrix (or Ferrers diagram;  see e.g. \cite[p62]{mp}) of $D$ is an $n \times n$ matrix $F$ of $\circ$'s, $\bullet$'s
and, $+$'s  such that
(i). All the diagonal entries and no others are $+$;  (ii).  For each $i,~i\in[n]$, the number of $\bullet$'s contained in the $i$th row is $d_i$;
(iii).  The signals $\bullet$'s in each row are to the left.
%\vspace{2mm}
%\end{defi}
Moreover,  the Ferrers matrix $F(G)$ of a  graph $G$ is defined to be the Ferrers matrix of the degree sequence $D(G)$ of $G$.  If $G$ is a threshold graph, it is easy to see that the adjacency matrix of $G$ is obtained from the Ferrers matrix of $G$  by  replacing symbols  $\circ$ and $+$ with   $0$, and replacing the symbol $\bullet$    with   $1$.  So  the  Ferrers matrix of a threshold graph is symmetric.   For example, the Ferrers matrix of a threshold graph $S_{6,9}$ is  symmetric  (see Figure \ref{fig:1}), the Ferrrers matrix of a non-threshold graph $G_{6,9}$ (see  Figure \ref{fig:2}) is  asymmetric.
 Brualdi and Hoffman in \cite{bruadih} defined a class of matrices in studying the spectral radius of graphs of order $n$  with size $m$.
  %We only concern and describe the entries below the diagonal of the Ferrers matrix.\vspace{2mm}

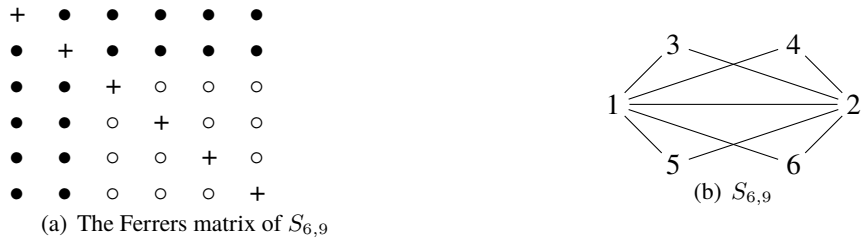
\begin{figure}[htbp]
\centering
\subfigure[The Ferrers matrix of $S_{6,9}$]{
\begin{minipage}{7cm}
\centering
\begin{tabular}{lllllllll}
+           & $\bullet$   & $\bullet$   & $\bullet$   & $\bullet$   & $\bullet$   &             &           &           \\
$\bullet$   & +           & $\bullet$   & $\bullet$   & $\bullet$   & $\bullet$   &             &           &           \\
$\bullet$   & $\bullet$   & +           & $\circ$     & $\circ$     & $\circ$     &             &           &           \\
$\bullet$   & $\bullet$ & $\circ$   & +           & $\circ$     & $\circ$     &             &           &           \\
$\bullet$   & $\bullet$   & $\circ$     & $\circ$     & +           & $\circ$     &             &           &           \\
$\bullet$ & $\bullet$ & $\circ$     & $\circ$     & $\circ$     & +           &             &           &           \\
\end{tabular}
%\caption{Ferrer matrix of $S(6,9)$}
\end{minipage}
}
\subfigure[ $S_{6,9}$]{
\begin{minipage}{7cm}
\centering
\begin{tikzpicture}[scale=.8]
        \node[fill=white,circle,inner sep=1pt] (t1) at (-2,0) {1};
        \node[fill=white,circle,inner sep=1pt] (t2) at (2,0) {2};
		\node[fill=white,circle,inner sep=1pt] (t3) at (-1,1) {3};
        \node[fill=white,circle,inner sep=1pt] (t4) at (1,1) {4};
        \node[fill=white,circle,inner sep=1pt] (t5) at (-1,-1) {5};
        \node[fill=white,circle,inner sep=1pt] (t6) at (1,-1) {6};

        \draw (t1)--(t2) (t1)--(t3) (t1)--(t4) (t1)--(t5) (t1)--(t6);
        \draw  (t2)--(t6) (t2)--(t4) (t2)--(t3) (t2)--(t5);
        \end{tikzpicture}
%\caption{Graph S(6,9)}
\end{minipage}
}
\caption{The Ferrers matrix of a threshold graph is symmetric}
\label{fig:1}
\end{figure}

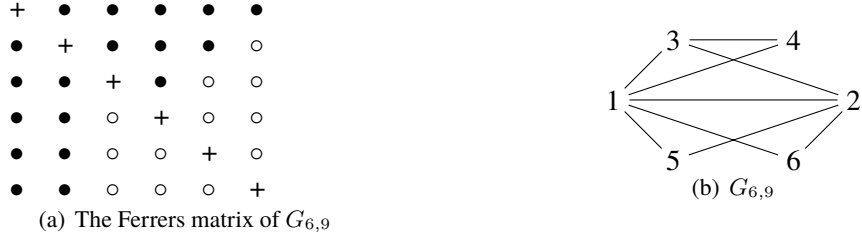
\begin{figure}[htbp]
\centering
\subfigure[The Ferrers matrix of $G_{6,9}$]{
\begin{minipage}{7cm}
\centering
\begin{tabular}{lllllllll}
+           & $\bullet$   & $\bullet$   & $\bullet$   & $\bullet$   & $\bullet$   &             &           &           \\
$\bullet$   & +           & $\bullet$   & $\bullet$  & $\bullet$   & $\circ$ &             &           &           \\
$\bullet$   & $\bullet$   & +           & $\bullet$    & $\circ$     & $\circ$     &             &           &           \\
$\bullet$   & $\bullet$ & $\circ$   & +           & $\circ$     & $\circ$     &             &           &           \\
$\bullet$   & $\bullet$   & $\circ$     & $\circ$     & +           & $\circ$     &             &           &           \\
$\bullet$ & $\bullet$   & $\circ$     & $\circ$     & $\circ$     & +           &             &           &           \\
\end{tabular}
%\caption{Ferrer matrix of $S(6,9)$}
\end{minipage}
}
\subfigure[ $G_{6,9}$]{
\begin{minipage}{7cm}
\centering
\begin{tikzpicture}[scale=.8]
        \node[fill=white,circle,inner sep=1pt] (t1) at (-2,0) {1};
        \node[fill=white,circle,inner sep=1pt] (t2) at (2,0) {2};
		\node[fill=white,circle,inner sep=1pt] (t3) at (-1,1) {3};
        \node[fill=white,circle,inner sep=1pt] (t4) at (1,1) {4};
        \node[fill=white,circle,inner sep=1pt] (t5) at (-1,-1) {5};
        \node[fill=white,circle,inner sep=1pt] (t6) at (1,-1) {6};

        \draw (t1)--(t2) (t1)--(t3) (t1)--(t4) (t1)--(t5) (t3)--(t4);
        \draw  (t1)--(t6) (t2)--(t6) (t2)--(t3) (t2)--(t5);
        \end{tikzpicture}
%\caption{Graph S(6,9)}
\end{minipage}
}
\caption{The Ferrers matrix of a non-threshold graph $G_{6,9}$ is asymmetrical}
\label{fig:2}
\end{figure}

\begin{defi}\cite{bruadih}
Let $A=(a_{ij})_{n\times n}$ be a $(0,1)$ matrix, and $tr(A)=0$. $A$ is said to be stepwise if it has the following property:
$$\mbox{ If }  h>k \mbox{ and } a_{hk}=1, \mbox{ then $ a_{ij}=1$   for all $j<  i\le h $ and $ j\le k$}.$$
\end{defi}
\noindent It is easy to see  that the following lemma holds.
\begin{lemma}\label{lemma-step}
A connected graph $G$ is threshold graph if and only if there exists a permutation matrix $P$ such that  $P^TA(G)P$ is a stepwise matrix.
\end{lemma}
\noindent Hence in the sequel,  for a threshold graph $G$, we always assume that the adjacency matrix of $G$ is a stepwise matrix with degree sequence $d_1 \ge d_2 \ge \cdots \ge d_n$. The following Lemma presents  some structure properties of graphs having maximum $A_\alpha$-spectral radius in $\mathcal{H}_{n,m}$.
%\begin{rem}
%If $G\in \mathcal{H}_{n, m}$, then $G$ is a threshold graph if and only if there exists a permutation matrix $P$ such that $P^TA(G)P$ is a stepwise adjacency matrix. Without loss of generality, for a threshold graph $G$, we assume that the adjacency matrix of $G$ is a stepwise adjacency matrix. So $d_1 \ge d_2 \ge \cdots \ge d_n$.
%\end{rem}
%Denote by $\mathcal{A}(n, m)=\{A(G) |  G\in \mathcal{G}_{n, m}\}$; $\mathcal{A}_c(n, m)=\{A(G) |  G\in \mathcal{H}_{n, m}\}$ and
%$\mathcal{A}^*(n, m)=\{A(G) |  A(G) \mbox{ is the stepwise adjacency  matrix of a threshold graph } G \in \mathcal{H}_{n, m}\}.$

\begin{lemma}\cite{lity}\label{threshold} For given two positive integers $n$ and $m\ge n-1$, and  $0\le \alpha <1$,  if a connected graph $G$ maximizes the $A_\alpha$-spectral radius in $\mathcal{H}_{n,m}$,  then $G$ must be  a threshold graph.
\end{lemma}

\begin{lemma} \label{uve} \cite{cfz,lity}  Let $u$, $v$ be distinct vertices of a connected graph $G$  and $x=(x_w)_{w\in V(G)}^T$ be the Perron vector of $A_{\alpha}(G)$ with $\alpha \in [0,1)$.\\
{\em (i).} If $N(u)\setminus\{v\} \supset N(v)\setminus\{u\}$, then $x_u >x_v$.\\
{\em (ii).} If $N(u)\setminus\{v\} = N(v)\setminus\{u\}$, then $x_u=x_v$.
\end{lemma}
\noindent It follows from Lemma \ref{uve} that we have the following corollary.
\begin{coro} \label{y1y2}
Let $G$ be a connected threshold graph with non-increasing degree sequence $(d_1, d_2, \ldots, d_n)$ and  $(y_1,y_2,\cdots,y_n)^T$ be the Perron vector of $A_{\alpha}(G)$, then\\
{\em (i).} $y_1 \ge y_2 \ge \cdots \ge y_n$. \\
{\em (ii).} If $d_i=d_j$, then $y_i=y_j$ for $1\le i, j\le n$.
\end{coro}

Let $M$ be a real symmetric $n\times n$ matrix, and let ${V}=\{1,2,\cdots,n\}=[n]$. Given a partition
$\Pi:{V}={V}_1 \bigcup {V}_2 \cdots \bigcup {V}_k$, the matrix $M$ can be correspondingly partitioned as
$$
\begin{pmatrix}
M_{1,1} & M_{1,2}& \cdots & M_{1,k}\\
M_{2,1} & M_{2,2}& \cdots & M_{2,k}\\
\vdots &\vdots & \ddots & \vdots\\
M_{k,1} & M_{k,2}& \cdots & M_{k,k}\\
\end{pmatrix}
$$
    The quotient matrix of $M$ with respect to $\Pi$ is defined as the $k \times k$ matrix
$M_{\Pi}=(b_{i,j})^k_{i,j=1}$ where $b_{i,j}$ is the average value of all row sums of $M_{i,j}$. The partition $\Pi$ is called equitable if each block $M_{i,j}$ of $M$ has constant row sum $b_{i,j}$. We also say that the quotient matrix $M_{\Pi}$ is equitable if $\Pi$ is an equitable partition of $M$.  The relationship between  eigenvalues of $M$ and $M_{\Pi}$ may be stated as follows.
\begin{lemma}\label{quot}  \cite{bh}
Let $M$ be a real symmetric matrix and $\lambda(M)$ be its largest eigenvalue. If $M_{\Pi}$
be an equitable quotient matrix of $M$, then the eigenvalues of $M_{\Pi}$ are also eigenvalues
of $M$. Furthermore, if $M$ is nonnegative and irreducible, then $\lambda(M)=\lambda(M_{\Pi})$.
\end{lemma}

\noindent In order to obtain our results, we also present some properties of the signless Laplacian spectral radius  $q(G)$ of the signless Laplacian matrix $Q(G)=D(G)+A(G)$ of a simple graph $G$.

\begin{lemma}\label{ue} \cite{das,feng}
If $G \in \mathcal{H}_{n,m}$, then
$$
q(G)=\max_{v \in V(G)} \{ d_v+\frac{1}{d_v}\sum\limits_{w \in N(v)} d_w \} \le \frac{2m}{n-1}+n-2.
$$
Moreover, $q(G)=\frac{2m}{n-1}+n-2$ if and only if $G$ is isomorphic to either $S_{n}$ or $K_n$.

\end{lemma}
\begin{lemma}\label{2n-2}
If $n\ge 4$, then $q(S_{n,2n-2}) \ge n+1.6$ and $q(S_{n,2n-1}) \ge n+1.75$.
\end{lemma}

\begin{proof}
Let $\Pi=\{ \{v_1,v_2\},\{v_3,v_4\},\{v_5,\cdots,v_n\}\}$. Then $Q(S_{n,2n-2})_{\Pi}$ is an
equitable quotient matrix of $Q(S_{n,2n-2})$.
Hence
\begin{equation*}
Q(S_{n,2n-2})_{\Pi}=\begin{pmatrix}
n &2 & n-4\\
2 &4 &0\\
2 &0  &2\\
\end{pmatrix}
\end{equation*}
Let $f(x)=det(xI_3-Q(S_{n,2n-2})_{\Pi})=x^{3} + \left(- n - 6\right) x^{2} + \left(4 n + 12\right) x - 24$. It is easy to see that the largest eigenvalue of $Q(S_{n,2n-2})_{\Pi}$ is greater than $n+1.6$.
 By Lemma \ref{quot}, $q(S_{n,2n-2}) \ge n+1.6$.

Let $\Pi=\{ \{v_1,v_2\},\{v_3\},\{v_4,v_5\},\{v_6,\cdots,v_n\}\}$. Then $Q(S_{n,2n-1})_{\Pi}$ is an equitable quotient matrix of $Q(S_{n,2n-1})$. Hence
\begin{equation*}
Q(S_{n,2n-1})_{\Pi}=\begin{pmatrix}
n &1& 2& n-5\\
2 &4& 2&0\\
2 &1 &3&0\\
2&0&0&2
\end{pmatrix}
\end{equation*}
Let $g(x)=det(xI_4-Q(S_{n,2n-1})_{\Pi})=x^{4} + \left(- n - 9\right) x^{3} + \left(7 n + 28\right) x^{2} + \left(- 10 n - 64\right) x + 72$. It is easy to see that $q(S_{n,2n-1}) \ge n+1.75$ by Lemma \ref{quot}.
\end{proof}

\section{Transformations}

In this section, we introduce several transformations and their properties, which will play a key role in the proof of the main results.
 Denote by $E_{pq}=(e_{ij})_{n\times n}$  the  $(0, 1)$- matrix with $e_{pq}=e_{qp}=1$, and $0$ in the other positions.
  %Denote by $\mathcal{A}(n, m)=\{A(G) |  G\in \mathcal{G}_{n, m}\}$; $\mathcal{A}_c(n, m)=\{A(G) |  G\in \mathcal{H}_{n, m}\}$ and
%$\mathcal{A}^*(n, m)=\{A(G) |  A(G) \mbox{ is the stepwise adjacency  matrix of a threshold graph } G \in \mathcal{H}_{n, m}\}.$

\begin{defi}\label{def0}
Let $A(G)=(a_{ij})_{n\times n}$ %\in \mathcal{A}^*(n,m)$
be the stepwise adjacency matrix of a connected threshold graph $G$.
We say the graph $G'$ is obtained from $G$ by a Transformation $( p, q; h, k)$, if there exist four positive integers $p$, $q$, $h$, $k$ such that the following conditions hold: \\
{\em (i).}   $2\le q<k<h<p$.\\
{\em (ii).}  $a_{pq}=0$, $a_{pj}=1$ whenever $j<q$; $a_{iq}=1$ whenever $q<i<p$.\\
{\em (iii).}  $a_{hk}=1$, $a_{hj}=0$ whenever $j>k$; $a_{ik}=0$ whenever $i>h$.\\
{\em (iv).}  $G'=G-v_hv_k+v_pv_q$.
\end{defi}
\noindent Clearly, $G'$ is also a connected threshold graph and $A(G')=A(G)-E_{hk}+E_{pq}$.

\begin{lemma}\label{rho}
Let $A(G)=(a_{ij})_{n\times n}$ %\in \mathcal{A}^*(n,m)$
be the stepwise adjacency matrix of a connected threshold graph $G$ and the graph $G^{\prime}$ be obtained from $G$ by a Transformation $(p,q; h,k)$.
If  $x=(x_1,x_2,\cdots,x_n)^T$ and $y=(y_1,y_2,\cdots,y_n)^T$ are the Perron vectors corresponding to $\rho_1=\rho_{\alpha}(G)$ and $\rho_2=\rho_{\alpha}(G')$, respectively, then %here $G'$ is obtained by a Transformation $(p, q; h, k)$ from $G$. We have
\begin{equation}\label{rho1}
(\rho_1-k\alpha)(x_h-x_p) =(k-q+1)\alpha x_p+(1-\alpha)(x_{q}+\cdots+x_{k}),
\end{equation}
\begin{equation}\label{rho2}
(\rho_2-p\alpha+1 )(y_q-y_k) =(p-h+1)\alpha y_k+(1-\alpha)(y_h+\cdots+y_{p}).
\end{equation}
\end{lemma}

\begin{proof}
By the $h$-th and $p$-th equations of $\rho_{\alpha}(G)x=A_{\alpha}(G)x$, we have
\begin{subequations}
\begin{align}
\rho_1x_h &=k\alpha x_h+(1-\alpha)(x_1+x_2+x_3+\cdots+x_{k}),\label{Zc}\\
\rho_1x_p &= (q-1)\alpha x_p+(1-\alpha)(x_1+x_2+x_3+\cdots+x_{q-1}).\label{Zd}
\end{align}
\end{subequations}
By subtracting \eqref{Zd} from \eqref{Zc},
\begin{equation}
\rho_1(x_h-x_p) =k\alpha x_h-(q-1)\alpha x_p+(1-\alpha)(x_{q}+\cdots+x_{k}),
\end{equation}
which implies
\begin{equation*}
(\rho_1-k\alpha)(x_h-x_p) =(k-q+1)\alpha x_p+(1-\alpha)(x_{q}+\cdots+x_{k}),
\end{equation*}
i.e. (1) holds.

By the $q$-th and $k$-th equations of $\rho_{\alpha}(G')y=A_{\alpha}(G')y$, we have
\begin{subequations}
\begin{align}
\rho_2y_q &= p \alpha y_q-y_q+(1-\alpha)(y_1+y_2+y_3+\cdots+y_{p}),\label{Ya}
\\\rho_2y_k &= (h-1)\alpha y_k-y_k+(1-\alpha)(y_1+y_2+y_3+\cdots+y_{h-1}).\label{Yb}
\end{align}
\end{subequations}
By subtracting \eqref{Yb} from \eqref{Ya},
\begin{equation}
\rho_2(y_q-y_k) =p \alpha y_q-(h-1)\alpha y_k-(y_q-y_k)+(1-\alpha)(y_h+\cdots+y_{p}),
\end{equation}
which implies
\begin{equation*}
(\rho_2-p\alpha+1 )(y_q-y_k) =(p-h+1)\alpha y_k+(1-\alpha)(y_h+\cdots+y_{p}),
\end{equation*}
i.e. (2) holds.
\end{proof}

\begin{lemma}\label{trij}
Let $G'$ be a connected threshold graph obtained from a connected threshold graph $G$ with degree sequence $(d_G(v_1),$ $ d_G(v_2), \cdots, d_G(v_n))$ by a Transformation $(p,q; h,k)$ with $k=q+1$.
If $\alpha \in [1/2,1)$, then $\rho_{\alpha}(G') \ge \rho_{\alpha}(G)$,
with equality if and only if $\alpha=1/2$ and $p=h+1=q+3$.
\end{lemma}

\begin{proof}
Denote $t=p+q-h-k=p-h-1\geq 0$, $\rho_1=\rho_{\alpha}(G), \rho_2=\rho_{\alpha}(G')$. Let $x=(x_1, x_2, \cdots, x_n)^T$ and $y=(y_1, y_2, \cdots, y_n)^T$ be the Perron vectors of $G$ and
$G'$, respectively. By Corollary \ref{y1y2}, we have $x_1\ge x_2 \ge \cdots \ge x_n>0$ and $y_1 \ge y_2 \ge \cdots \ge y_n>0$. By $k=q+1$, we have $d_{G'}(v_h)=d_{G'}(v_{h+1})=\cdots=d_{G'}(v_p)$.
Furthermore, by $d_{G'}(v_h)=d_{G'}(v_{h+1})=\cdots=d_{G'}(v_p)$ and  (ii) in Corollary~\ref{y1y2},
 \begin{equation}\label{XX1}
 y_h=y_{h+1}=\cdots=y_p.
 \end{equation}
By the definition of Transformation $(p,q; h,k)$, $2\le q<k<h<p$, which implies that $y_k \ge y_h$. Hence by $\alpha \in [1/2,1)$, we have
\begin{equation}\label{XX2}
\alpha y_k+(1-\alpha)y_h \ge \alpha y_h+(1-\alpha)y_k,
\end{equation}
with equality if and only if $\alpha=\frac{1}{2}$ or $h=k+1$.
%(If $y_h=y_k$, then $d_{G'}(v_k)=d_{G'}(v_h)= k-1$, so $d_{G}(v_k)=h-1=k$).
 Similarly by $x_q \ge x_p$, we have
\begin{equation}\label{XX2-1}
\alpha x_q+(1-\alpha)x_p \ge \alpha x_p+(1-\alpha)x_q,
\end{equation}
with equality if and only if $\alpha=\frac{1}{2}$.
%(If $x_q=x_p$, then $d_{G}(v_q)=d_{G}(v_p)= q-1$, which contradicts the fact that $d_{G}(v_q)=p-2>q-1$).
On the one hand, by (2) in Lemma~\ref{rho} and $y_h=y_{h+1}=\cdots=y_p$,
\begin{equation}\label{XXa}
(\rho_{2}-p\alpha+1)(y_q-y_k)=(p-h+1)\alpha y_k+(1-\alpha)(y_h+\cdots+y_p)=(p-h+1)(\alpha y_k+(1-\alpha)y_h).
\end{equation}
On the other hand, by (1) in Lemma~\ref{rho} and $x_q\ge x_{q+1} \ge \cdots \ge x_k$,
\begin{equation}\label{XXb}
(\rho_{1}-k\alpha)(x_h-x_p)=(k-q+1)\alpha x_p+(1-\alpha)(x_q+\cdots+x_k)\le(k-q+1)(\alpha x_p+(1-\alpha)x_q).
\end{equation}
In addition,
\begin{equation}\label{XXc}
x^T A_{\alpha}(G)y=\alpha \sum \limits_{v_iv_j \in E(G)}(x_iy_i+x_jy_j)+
(1-\alpha)\sum\limits_{v_iv_j \in E(G)}(x_iy_j+x_jy_i).
\end{equation}
Hence, by $x_q \ge x_k$ and by \eqref{XX1}-%, \eqref{XX2}, \eqref{XX2-1}, \eqref{XXa}, \eqref{XXb} and
 \eqref{XXc},
\begin{equation}\label{eqgij}
 \begin{array}{lll}
&&x^T(\rho_2-\rho_1)y\\
&=&x^T(A_{\alpha}(G')-A_{\alpha}(G))y\\
&=&\alpha(x_py_p+x_qy_q-x_hy_h-x_ky_k)+(1-\alpha)(x_py_q+x_qy_p-x_hy_k-x_ky_h)\\
&=&(x_q-x_k)(\alpha y_k+(1-\alpha)y_h)+(y_q-y_k)(\alpha x_q+(1-\alpha)x_p)\\
&&+(y_p-y_h)(\alpha x_p+(1-\alpha)x_q)+(x_p-x_h)(\alpha y_h+(1-\alpha)y_k)\\
&\ge&(y_q-y_k)(\alpha x_q+(1-\alpha)x_p)+(x_p-x_h)(\alpha y_h+(1-\alpha)y_k) \\
&\ge &(\alpha x_q+(1-\alpha)x_p)(\alpha y_k+(1-\alpha)y_h)\frac{p-h+1}{\rho_2-p\alpha+1}\\
&&-(\alpha y_h+(1-\alpha)y_k)(\alpha x_p+(1-\alpha)x_q)\frac{k-q+1}{\rho_1-k\alpha}\\
&\ge&(\alpha x_q+(1-\alpha)x_p)(\alpha y_k+(1-\alpha)y_h)[\frac{p-h+1}{\rho_2-p\alpha+1}-\frac{k-q+1}{\rho_1-k\alpha}]\\
&=&(\alpha x_q+(1-\alpha)x_p)(\alpha y_k+(1-\alpha)y_h)\frac{(p-h+1)(\rho_1-k\alpha)-(k-q+1)
(\rho_2-p\alpha+1)}{(\rho_2-p\alpha+1)(\rho_1-k\alpha)}\\
&=&(\alpha x_q+(1-\alpha)x_p)(\alpha y_k+(1-\alpha)y_h)\frac{(k-q+1)(\rho_1-\rho_2)+(k-q+1)
(p\alpha-1-k\alpha)+t(\rho_1-k\alpha)}{(\rho_2-p\alpha+1)(\rho_1-k\alpha)}.\\
\end{array}
\end{equation}
 Hence,  \eqref{eqgij} yields
 %\begin{small}
\begin{equation}
 \begin{array}{lll}
&&(\rho_2-\rho_1)(x^Ty\frac{(\rho_2-p\alpha+1)(\rho_1-k\alpha)}{(\alpha x_q+(1-\alpha)x_p)(\alpha y_k+(1-\alpha)y_h)}+k-q+1)\\
 &\ge&(k-q+1)(p\alpha-1-k\alpha)+t(\rho_1-k\alpha)\\
 &\ge& 0.
\end{array}
\end{equation}
%\end{small}
It follows that $\rho_2\ge\rho_1$.

Moreover, $\rho_2=\rho_1$ holds if and only if $x_q=x_k$, $t=0$, $p\alpha-1-k\alpha=0$ and $\alpha=\frac{1}{2}$, therefore, $p=k+2=h+1=q+3$ and $\alpha=\frac{1}{2}$.
\end{proof}

\begin{lemma}\label{tr24}
Let $G'$ be the graph obtained from a connected threshold graph $G$ by a Transformation $(p,q;h,k)$ with $k=q+2$. If $\alpha \in [1/2,1)$ and $p>h+1$, then  $\rho_{\alpha}(G') > \rho_{\alpha}(G)$.
\end{lemma}

\begin{proof}
Let $d_G(v_1)\geq d_G(v_2)\geq \cdots\geq d_G(v_n)$ be the degree sequence of a connected threshold graph $G$.
Clearly $p>d_G(v_{q+1})+1 \ge h$. Denote by $p_1=d_G(v_{q+1})+1$. We consider the following two cases:

{\bf Case 1: $d_G(v_{q+1})<p-2$. }

Clearly $d_G(v_{q+1})\ge h-1$ and $p \ge d_G(v_{q+1})+3$. Let $G_1$ be the graph obtained from $G$ by a Transformation $(p_1+1,q+1;h,q+2)$. Hence, $G'$ can be obtained from $G_1$ by a Transformation $(p,q;p_1+1,q+1)$. By Lemma \ref{trij}, $\rho_{\alpha}(G') \ge \rho_{\alpha}(G_1)$ and $\rho_{\alpha}(G_1) \ge \rho_{\alpha}(G)$. Furthermore, by Lemma \ref{trij}, $\rho_{\alpha}(G') = \rho_{\alpha}(G)$ implies $p_1+1=q+4$ and  $p_1+1=q+2$, which is a contradiction. Therefore, $\rho_{\alpha}(G') >\rho_{\alpha}(G)$.

{\bf Case 2: $d_G(v_{q+1})=p-2$. }

Clearly $d_G(v_q)=p-2$, $d_G(v_{q+1})=p-2$ and $p\ge h+2$. Let $G_1$ be the graph from  $G$ by a Transformation $(p,q;p-1,q+1)$. Hence, $G_1$ is a connected threshold graph in $\mathcal{H}_{n, m}$. By Lemma \ref{trij},  $\rho_{\alpha}(G_1) \ge \rho_{\alpha}(G)$.
Then it is easy to see that $G'$ is the graph also obtained from $G_1$ by a Transformation $(p-1,q+1;h,q+2)$. Hence by Lemma \ref{trij}, $\rho_{\alpha}(G') \ge \rho_{\alpha}(G_1)\ge \rho_{\alpha}(G)$.  Furthermore, by Lemma \ref{trij}, $\rho_{\alpha}(G') = \rho_{\alpha}(G)$ implies $p=q+3$ and $p-1=q+4$, which is a contradiction. Therefore, $\rho_{\alpha}(G') >\rho_{\alpha}(G)$.
\end{proof}

\noindent In order to prove the main results, we also introduce  the other two general transformations:

\begin{defi}\label{def}
Let $A(G)=(a_{ij})_{n\times n}$ %\in \mathcal{A}^*(n,m)$
 be the stepwise adjacency matrix of a connected threshold graph $G$.
 % and $l$ be a nonnegative integer.
We say the graph $G'$ is obtained from $G$ by a Transformation $(p,q;h,k;l+1,1)$, if there exist four positive integers $p$, $q$, $h$, $k$ and one nongective integer $l$ such that the following conditions holds:  \\
{\em (i).}  $q<k \le k+l<h<p-l$.\\
{\em (ii).}  $a_{iq}=0$, $a_{ij}=1$ whenever for all integers $p-1\le i\le p$, $j<q$; $a_{iq}=1$ whenever $i<p-l$.\\
{\em (iii).}  $a_{hi}=1$, $a_{hj}=0$, $a_{h+1,i}=0$ whenever for all integers $k\le i\le k+l]$, $j>k+l$.\\
{\em (iv).}  $$G'=G-\sum\limits_{j=0}^{l} v_{h}v_{k+j}+\sum\limits_{j=0}^{l} v_{p-j}v_{q}.$$
\end{defi}
\noindent Clearly, $G'$ is also a connected threshold graph and
$$A(G')=A(G)-\sum\limits_{j=0}^{l} E_{h,k+j}+\sum\limits_{j=0}^{l} E_{p-j,q}.$$
\begin{rem}
%Clearly $A'$  can be  obtained from $A$ by interchanging the $(p-s,q)$ and $(h,k+s)$ entries and
%interchanging the $(q,p-s)$ and $(k+s,h)$ entries for all integers $s \in[0,l]$.
If $l=0$, then the Transformation $(p,q;h,k;l+1,1)$ is actually the Transformation $(p,q;h,k)$.
\end{rem}

\begin{rem}
Let $G$ be a connected threshold graph with the non-increasing degree sequence $(d_1, \ldots, d_n)$. If $G'$ is obtained from $G$ by a Transformation $(p,q;h,k;l+1,1)$, then the non-increasing degree sequence $d'=(d'_1, \ldots, d'_n)$ of $G'$ is the same as  the degree sequence of $G$ except $d'_{q}=p-1$, $d'_{h}=k-1$, $d'_{k}=\cdots=d'_{k+l}=h-2$ and $d'_{p}=\cdots=d'_{p-l}=q$.
 %then we say $G'$ can be obtained by Transformation$(p,q;h,k;l+1,1)$ on $G$). we also call this transformation as Transformation from $G$ to $G'$ with respect to $(p,q;h,k;l+1,1)$.
\end{rem}

%Obviously,
%%Assume $A=[a_{ij}]$ and $A'=[a'_{ij}]$ are the stepwise adjacency matrix of $G$ and $G'$ respectively, then we have
%\\ {\em (1)}. $q<k$, $k+l<h$, $h<p$, and
%\\ {\em (2).} $a_{sq}=1$, $a_{sj}=0$ whenever $j>q$, $a_{iq}=0$ whenever $i>p+l$, for all integers $s \in [p,p+l]$, and
%\\ {\em (3).} $a_{hs}=0$, $a_{hj}=1$ whenever $j<k$, $a_{is}=1$ whenever $s<i<h$, , for all integers $s \in [j,j+l]$.

\noindent For example, there is a concrete transformation  from $L_{7, 12}$ to $S_{7, 12}$ with respect to $(7,2;5,3; 2, 1)$ which  is depicted in Figure  \ref{tt3}.
\begin{figure}[htbp]
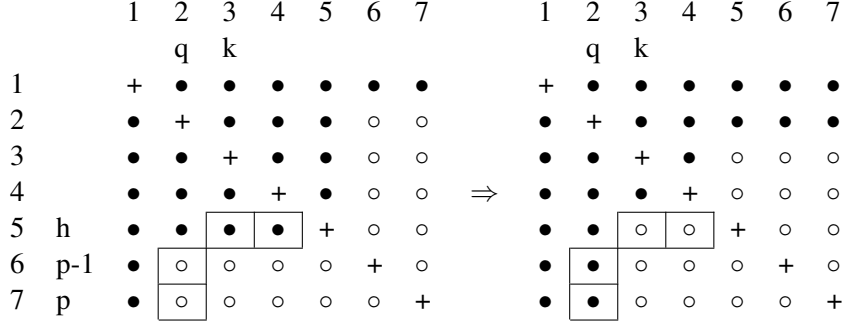

\centering
\begin{tabular}{lllllllllllllllll}
  &  &1          &  2 &  3  &  4 & 5 & 6 & 7 & &         1 &   2 &  3 &  4 &5  &6  & 7 \\
    &  &          &  q &  k  &   &  &  &  & &          &   q &  k &   &  &  &  \\
1 & & +           & $\bullet$   & $\bullet$   & $\bullet$   & $\bullet$ & $\bullet$ & $\bullet$  & \multirow{7}{*}{ $\Rightarrow $ } & +           & $\bullet$   & $\bullet$   & $\bullet$   & $\bullet$ & $\bullet$ & $\bullet$ \\
2 & & $\bullet$   & +           & $\bullet$   & $\bullet$   & $\bullet$ & $\circ$ & $\circ$  & &$\bullet$   & +           & $\bullet$   & $\bullet$   & $\bullet$ & $\bullet$   & $\bullet$  \\
3 & & $\bullet$   & $\bullet$   & +           & $\bullet$   & $\bullet$   & $\circ$   & $\circ$  & & $\bullet$   & $\bullet$   & +           & $\bullet$   & $\circ$ & $\circ$   & $\circ$ \\
4 & & $\bullet$   & $\bullet$   & $\bullet$   & +           & $\bullet$   & $\circ$   & $\circ$    & & $\bullet$   & $\bullet$   & $\bullet$   & +           & $\circ$ & $\circ$   & $\circ$ \\ \cline{5-6}  \cline{13-14}
5 &h & $\bullet$   & \multicolumn{1}{l|}{$\bullet$} & \multicolumn{1}{l|}{$\bullet$}   & \multicolumn{1}{l|}{$\bullet$}   & +         & $\circ$   & $\circ$  & & $\bullet$   & \multicolumn{1}{l|}{$\bullet$} & \multicolumn{1}{l|}{$\circ$} & \multicolumn{1}{l|}{$\circ$} & +         & $\circ$   & $\circ$  \\ \cline{4-6}  \cline{12-14}
6 &p-1 & \multicolumn{1}{l|}{$\bullet$} & \multicolumn{1}{l|}{$\circ$} & $\circ$     & $\circ$     & $\circ$   & +         & $\circ$  & & \multicolumn{1}{l|}{$\bullet$} & \multicolumn{1}{l|}{$\bullet$}   & $\circ$     & $\circ$     & $\circ$   & +         & $\circ$    \\ \cline{4-4}
\cline{12-12}
7 & p& \multicolumn{1}{l|}{$\bullet$} & \multicolumn{1}{l|}{$\circ$} & $\circ$     & $\circ$     & $\circ$   & $\circ$   & +    & & \multicolumn{1}{l|}{$\bullet$} & \multicolumn{1}{l|}{$\bullet$}   & $\circ$     & $\circ$     & $\circ$   & $\circ$   & +        \\ \cline{4-4}
\cline{12-12}
\end{tabular}
\caption{Transformation $(7, 2; 5, 3; 2,1)$ from $L_{7, 12}$ to $S_{7, 12}$}
%\end{table}
\label{tt3}
\end{figure}

\begin{lemma}\label{trg}
 Let $G'$ be the graph obtained from a connected threshold graph $G$ by a Transformation $(p,q;h,k;l+1,1)$. If $\alpha \in [1/2,1)$ and $k=q+1$, then $\rho_{\alpha}(G') \ge \rho_{\alpha}(G)$, with equality if and only if $\alpha=1/2$, $l=0$ and $p=h+1=q+3$.
\end{lemma}

\begin{proof}
Let $\rho_1=\rho_{\alpha}(G)$ and $ \rho_2=\rho_{\alpha}(G')$.
By Definition \ref{def} and $k=q+1$,  we have $d_{G'}(v_h)=d_{G'}(v_p)$.
By Corollary \ref{y1y2}, we have  $x_1\ge x_2 \ge \cdots \ge x_n>0$ and $y_1 \ge y_2 \ge \cdots \ge y_n>0$.
By $d_G(v_k)=d_G(v_{k+1})=\cdots=d_G(v_{k+l})$,  we have  $x_{k}=x_{k+1}=\cdots=x_{k+l}$.
By $d_G(v_{p-l})=d_G(v_{p-l+1})=\cdots=d_G(v_{p})$, we have  $x_{p}=x_{p-1}=\cdots=x_{p-l}$.
Moreover, $d_{G'}(v_k)=d_{G'}(v_{k+1})=\cdots=d_{G'}(v_{k+l})$ implies that
$y_{k}=y_{k+1}=\cdots=y_{k+l}$.
$d_{G'}(v_{p})=d_{G'}(v_{p-1})=\cdots=d_{G'}(v_{h})$ implies that $y_{p}=y_{p-1}=\cdots=y_{h}$.

%In addition,
%\begin{equation}\label{eqgg2}
%x^T A_{\alpha}(G)y=\alpha \sum \limits_{v_iv_j \in E(G)}(x_iy_i+x_jy_j)+
%(1-\alpha)\sum\limits_{v_iv_j \in E(G)}(x_iy_j+x_jy_i),
%\end{equation}

Hence, by  \eqref{XXc},
\begin{equation}\label{eqgl1}
 \begin{array}{lll}
&&x^T(\rho_2-\rho_1)y\\
&=&x^T(A_{\alpha}(G')-A_{\alpha}(G))y\\
&=&\sum\limits_{j=0}^l[\alpha(x_{p-j}y_{p-j}+x_qy_q-x_hy_h-x_{k+j}y_{k+j})\\
&& +(1-\alpha)(x_{p-j}y_q+x_qy_{p-j}-x_hy_{k+j}-x_{k+j}y_h)]\\
&=&(l+1)[\alpha(x_py_p+x_qy_q-x_hy_h-x_ky_k)+(1-\alpha)(x_py_q+x_qy_p-x_hy_k-x_ky_h)]\\
&=&(l+1)[(x_q-x_k)(\alpha y_k+(1-\alpha)y_h)+(y_q-y_k)(\alpha x_q+(1-\alpha)x_p)\\
&&+(y_p-y_h)(\alpha x_p+(1-\alpha)x_q)+(x_p-x_h)(\alpha y_h+(1-\alpha)y_k)] \\
&\ge& (l+1)[(y_q-y_k)(\alpha x_q+(1-\alpha)x_p)+
(x_p-x_h)(\alpha y_h+(1-\alpha)y_k)]\\
&\ge& (l+1)[\frac{p-h+1}{\rho_2-\alpha p+1}(\alpha y_k+(1-\alpha)y_h)(\alpha x_q+(1-\alpha)x_p)\\
&&-\frac{k+l-q+1}{\rho_1-\alpha(k+l)}(\alpha x_p+(1-\alpha)x_q)(\alpha y_h+(1-\alpha)y_k)]\\
&\ge&  (l+1)(\alpha y_k+(1-\alpha)y_h)(\alpha x_q+(1-\alpha)x_p)(\frac{p-h+1}{\rho_2-\alpha p+1}-\frac{k+l-q+1}{\rho_1-\alpha(k+l)})\\
&=&  (l+1)(\alpha y_k+(1-\alpha)y_h)(\alpha x_q+(1-\alpha)x_p)\frac{(p-h+1)(\rho_1-\alpha(k+l))-(k+l-q+1)(\rho_2-\alpha p+1)}{(\rho_2-\alpha p+1)(\rho_1-\alpha(k+l))}\\
%&=&  (l+1)(\alpha y_k+(1-\alpha)y_h)(\alpha x_q+(1-\alpha)x_p)\frac{(p-h+1)(\rho_1-\alpha(k+l))-(k+l-q+1)(\rho_2-\alpha p+1)}{(\rho_2-\alpha p+1)(\rho_1-\alpha(k+l))}\\
&=&  (l+1)(\alpha y_k+(1-\alpha)y_h)(\alpha x_q+(1-\alpha)x_p)\\
&&\times \frac{(k+l-q+1)(\rho_1-\rho_2)+(k+l-q+1)(\alpha(p-k-l)-1)+(p+q-h-k-l)(\rho_1-\alpha(k+l))}{(\rho_2-\alpha p+1)(\rho_1-\alpha(k+l))}.
\end{array}
\end{equation}

%The right of Equation \eqref{eqgg1} is similar to Equation \eqref{eqgij}, hence
%the rest of the proof is similar to Lemma \ref{trij}.

Hence by \eqref{eqgl1}, we have
\begin{eqnarray*}
&&(\rho_2-\rho_1)[x^Ty\frac{(\rho_2-\alpha p+1)(\rho_1-\alpha(k+l))}{(l+1)(\alpha y_k+(1-\alpha)y_h)(\alpha x_q+(1-\alpha)x_p)}+k+l-q+1]\\
&\ge& (k+l-q+1)(\alpha(p-k-l)-1)+(p+q-h-k-l)(\rho_1-\alpha(k+l))\\
&\geq& 0,
\end{eqnarray*} which implies $\rho_2\geq\rho_1$.

Further, $\rho_2=\rho_1$ holds if and only if $x_q=x_k$, $p+q-h-k-l=0$, $\alpha(p-k-l)-1=0$ and $\alpha=\frac{1}{2}$, therefore, $l=0$, $p=k+2=h+1=q+3$ and $\alpha=\frac{1}{2}$.
\end{proof}

\noindent We also need the following transformation.
\begin{defi}\label{def2}
Let $A(G)=(a_{ij})_{n\times n}$ %\in \mathcal{A}^*(n,m)$
be the stepwise adjacency matrix of a connected threshold graph $G$.
We say the graph $G'$ is obtained from $G$ by a Transformation $(p,q;h,k;1,l+1)$, if there exist four positive integers $p$, $q$, $h$, $k$  and one nonnegative integer $l$ such that the following conditions holds:  \\
{\em (i).}  $2\le q-l \le q<k <h-l \le h<p$.\\
{\em (ii).}  $a_{ps}=0$, $a_{pj}=1$ whenever $j<q-l$ and $a_{is}=1$ whenever $s<i<p$, for all integers $q-l\le s\le q$.\\
{\em (iii).}  $a_{sk}=1$, $a_{sj}=0$ whenever $j>k$ and $a_{ik}=0$ whenever $i>h$, for all integers $h-l\le s\le h$.\\
{\em (iv).}  $$G'=G-\sum\limits_{j=0}^{l} v_{h-j}v_{k}+\sum\limits_{j=0}^{l} v_{p}v_{q-j}.$$
\end{defi}
Clearly, $G'$ is also a connected threshold graph and
$$A(G')=A(G)-\sum\limits_{j=0}^{l} E_{h-j,k}+\sum\limits_{j=0}^{l} E_{p,q-j}.$$
\noindent For example, $S_{9, 23}$ is the graph obtained from $G_{9, 23}$  by a Transformation $(9,3;8,4;1,2)$,  which is depicted in Figure  ~\ref{tt923}.

\begin{figure}[htbp]
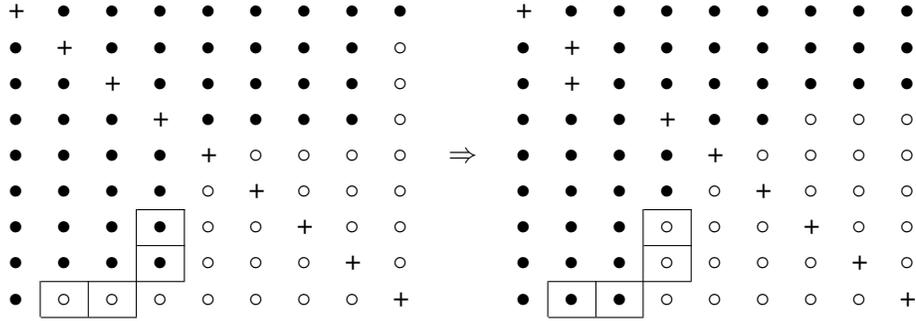

%\begin{table}[htbp]
\centering
\begin{tabular}{lllllllllllllllllll}
+           & $\bullet$   & $\bullet$   & $\bullet$   & $\bullet$ & $\bullet$ & $\bullet$ & $\bullet$ & $\bullet$
&\multirow{9}{*}{ $ \Rightarrow $ }  & +          & $\bullet$                      & $\bullet$   & $\bullet$   & $\bullet$ & $\bullet$ & $\bullet$ & $\bullet$ & $\bullet$  \\
$\bullet$   & +           & $\bullet$   & $\bullet$   & $\bullet$ & $\bullet$ & $\bullet$ & $\bullet$ &  $\circ$  & & $\bullet$  & +           & $\bullet$   & $\bullet$   & $\bullet$ & $\bullet$ & $\bullet$ & $\bullet$ & $\bullet$\\
$\bullet$   & $\bullet$   & +           & $\bullet$   & $\bullet$ & $\bullet$ & $\bullet$ & $\bullet$ & $\circ$  & &$\bullet$      & +           & $\bullet$   & $\bullet$ & $\bullet$ & $\bullet$ & $\bullet$ & $\bullet$  & $\bullet$\\
$\bullet$   & $\bullet$   & $\bullet$   & +           & $\bullet$ & $\bullet$ & $\bullet$   & $\bullet$   & $\circ$  & & $\bullet$   & $\bullet$   & $\bullet$   & +           & $\bullet$ & $\bullet$ & $\circ$  & $\circ$  & $\circ$\\
$\bullet$   & $\bullet$   & $\bullet$   & $\bullet$   & +         & $\circ$   & $\circ$   & $\circ$   & $\circ$  & & $\bullet$   & $\bullet$   & $\bullet$   & $\bullet$   & +         & $\circ$   & $\circ$   & $\circ$   & $\circ$   \\
$\bullet$   & $\bullet$   & $\bullet$   & $\bullet$   & $\circ$   & +         & $\circ$   & $\circ$   & $\circ$ & & $\bullet$   & $\bullet$   & $\bullet$   & $\bullet$   & $\circ$   & +         & $\circ$   & $\circ$   & $\circ$   \\ \cline{4-4} \cline{14-14}
$\bullet$   & $\bullet$   & \multicolumn{1}{l|}{$\bullet$} & \multicolumn{1}{l|}{$\bullet$}   & $\circ$   & $\circ$   & +         & $\circ$   & $\circ$  & & $\bullet$   & $\bullet$   & \multicolumn{1}{l|}{$\bullet$} & \multicolumn{1}{l|}{$\circ$} & $\circ$   & $\circ$   & +         & $\circ$   & $\circ$ \\ \cline{4-4} \cline{14-14}
$\bullet$   & $\bullet$   & \multicolumn{1}{l|}{$\bullet$} & \multicolumn{1}{l|}{$\bullet$}   & $\circ$   & $\circ$   & $\circ$   & +         & $\circ$  & &  $\bullet$   & $\bullet$   & \multicolumn{1}{l|}{$\bullet$} & \multicolumn{1}{l|}{$\circ$} & $\circ$   & $\circ$   & $\circ$   & +         & $\circ$  \\ \cline{2-4} \cline{12-14}
\multicolumn{1}{l|}{$\bullet$} & \multicolumn{1}{l|}{$\circ$} & \multicolumn{1}{l|}{$\circ$} & $\circ$     & $\circ$   & $\circ$   & $\circ$   & $\circ$   & +      & & \multicolumn{1}{l|}{$\bullet$} & \multicolumn{1}{l|}{$\bullet$}   & \multicolumn{1}{l|}{$\bullet$}   & $\circ$     & $\circ$   & $\circ$   & $\circ$   & $\circ$   & +        \\ \cline{2-3} \cline{12-13}
\end{tabular}
\caption{\mbox{Transformation} $ (9,3;8,4;1,2)$ from $G_{9,23}$ to $S_{9,23}$}
%\end{table}
\centering
\label{tt923}
\end{figure}

\begin{lemma}\label{trg1}
 Let $G'$ be obtained by a Transformation $(p,q;h,k;1,l+1)$ from a connected threshold graph $G$. If $\alpha \in [1/2,1)$ and $k=q+1$, then $\rho_{\alpha}(G') \ge \rho_{\alpha}(G)$, with equality if and only if $\alpha=1/2$, $l=0$, and $p=h+1=q+3$.
\end{lemma}

\begin{proof}
Let $\rho_1=\rho_{\alpha}(G), \rho_2=\rho_{\alpha}(G')$.
By Definition \ref{def2} and $k=q+1$, $d_{G'}(v_h)=d_{G'}(v_p)$.
By Corollary \ref{y1y2}, $x_1\ge x_2 \ge \cdots \ge x_n>0$ and $y_1 \ge y_2 \ge \cdots \ge y_n>0$.
By $d_G(v_{q-l})=d_G(v_{q-l+1})=\cdots=d_G(v_{q})$, $x_{q}=x_{q-1}=\cdots=x_{q-l}$.
By $d_G(v_{h-l})=d_G(v_{h-l+1})=\cdots=d_G(v_{h})$, $x_{h}=x_{h-1}=\cdots=x_{h-l}$.
Moreover, $d_{G'}(v_{q-l})=d_{G'}(v_{q-l+1})=\cdots=d_{G'}(v_{q})$ implies that
$y_{q-l}=y_{q-l+1}=\cdots=y_{q}$.
$d_{G'}(v_{p})=d_{G'}(v_{p-1})=\cdots=d_{G'}(v_{h-l})$ implies that $y_{p}=y_{p-1}=\cdots=y_{h-l}$. By \eqref{XXc} and similar to Lemma \ref{trij}, we have
\begin{equation}\label{eqg1l}
 \begin{array}{lll}
& &x^T(\rho_2-\rho_1)y\\
&=&x^T(A_{\alpha}(G')-A_{\alpha}(G))y\\
&=&\sum\limits_{j=0}^l[\alpha(x_{q-j}y_{q-j}+x_py_p-x_ky_k-x_{h-j}y_{h-j})\\
&&+(1-\alpha)(x_{q-j}y_p+x_py_{q-j}-x_ky_{h-j}-x_{h-j}y_k)]\\
&=&(l+1)[\alpha(x_py_p+x_qy_q-x_hy_h-x_ky_k)+(1-\alpha)(x_py_q+x_qy_p-x_hy_k-x_ky_h)]\\
&=&(l+1)[(x_q-x_k)(\alpha y_k+(1-\alpha)y_h)+(y_q-y_k)(\alpha x_q+(1-\alpha)x_p)\\
&&+(y_p-y_h)(\alpha x_p+(1-\alpha)x_q)+(x_p-x_h)(\alpha y_h+(1-\alpha)y_k)] \\
&\ge& (l+1)[(y_q-y_k)(\alpha x_q+(1-\alpha)x_p)+
(x_p-x_h)(\alpha y_h+(1-\alpha)y_k)]\\
&\ge& (l+1)[\frac{p-h+l+1}{\rho_2-\alpha p+1}(\alpha y_k+(1-\alpha)y_h)(\alpha x_q+(1-\alpha)x_p)\\
&&-\frac{k-q+l+1}{\rho_1-\alpha k}(\alpha x_p+(1-\alpha)x_q)(\alpha y_h+(1-\alpha)y_k)]\\
&\ge&(l+1)(\alpha y_k+(1-\alpha)y_h)(\alpha x_q+(1-\alpha)x_p)(\frac{p-h+l+1}{\rho_2-\alpha p+1}-\frac{k+l-q+1}{\rho_1-\alpha k})\\
&=&(l+1)(\alpha y_k+(1-\alpha)y_h)(\alpha x_q+(1-\alpha)x_p)\frac{(p-h+l+1)(\rho_1-\alpha k)-(k+l-q+1)(\rho_2-\alpha p+1)}{(\rho_2-\alpha p+1)(\rho_1-\alpha k)}\\
&=&(l+1)(\alpha y_k+(1-\alpha)y_h)(\alpha x_q+(1-\alpha)x_p)\\
&&\times \frac{(k+l-q+1)(\rho_1-\rho_2)+(k+l-q+1)(\alpha p-\alpha k -1)+(p+q-h-k)(\rho_1-\alpha k)}{(\rho_2-\alpha p+1)(\rho_1-\alpha k)}.\\
\end{array}
\end{equation}
Multiply both sides by $\frac{(\rho_2-\alpha p+1)(\rho_1-\alpha k)}{(l+1)(\alpha y_k+(1-\alpha)y_h)(\alpha x_q+(1-\alpha)x_p)}$, \eqref{eqg1l} can be rearranged to the following form:%A little calculation yields
\begin{equation}\label{egeg4}
 \begin{array}{lll}
& &(\rho_2-\rho_1)(\frac{x^Ty(\rho_2-\alpha p+1)(\rho_1-\alpha k)}{(l+1)(\alpha y_k+(1-\alpha)y_h)(\alpha x_q+(1-\alpha)x_p)}+k+l-q+1)\\
&\ge&(k+l-q+1)(\alpha p-\alpha k -1)+(p+q-h-k)(\rho_1-\alpha k)\\
&\ge & 0.
\end{array}
\end{equation}

Further, $\rho_2=\rho_1$ holds if and only if $x_q=x_k$, $p+q-h-k=0$, $\alpha(p-k)-1=0$ and $\alpha=\frac{1}{2}$, therefore, $l=0$, $p=k+2=h+1=q+3$ and $\alpha=\frac{1}{2}$.
%Using the similar method in Lemma \ref{trg}, we can obtain the results, and we omit the procedure here.
%%增加与引理3.4类似的讨论
\end{proof}

\section{Proof of Theorems \ref{t2n-2-1} and \ref{tknc}}

In this section, we first prove the following theorem, which extends the main results of  Li et.al \cite{lity} for $n-1\le m\le 2n-3$ and  Chang and Tam \cite{ct2} for $n-1\le m\le 2n-3$ and $\alpha=1/2$.

\begin{theo}\label{t2n-2c}
Let $n-1 \le m \le 2n-2$.\\
{\em (i).} If $\alpha \in (\frac{1}{2},1)$  or $m \neq n+2$ and $\alpha=\frac{1}{2}$, then $S_{n,m}$ is the only extremal graph that maximizes the $A_{\alpha}$-spectral radius in $\mathcal{H}_{n,m}$.\\
{\em (ii).} If $m = n+2$ and $\alpha=\frac{1}{2}$, then $S_{n,n+2}$ and $\tilde{S}_{n,n+2}$ are all extremal graphs that maximize the $A_{\frac{1}{2}}$-spectral radius in $\mathcal{H}_{n,n+2}$.
\end{theo}

\begin{proof}
Note that  $\rho_\alpha(\widetilde{S}_{n,n+2})=\rho_{\alpha}(S_{n,n+2})$.
It is sufficient to prove that if  $G'\neq S_{n,m}$  is a graph which has the maximum $A_{\alpha}$-spectral radius in $\mathcal{H}_{n,m}$ then $m=n+2$, $\alpha=1/2$ and $G'=\tilde{S}_{n,m}$.

 If $n\le 5, $ it is easy to see that the assertion holds. So we assume that $n\ge 6$.   By Lemma \ref{threshold},  we assume that $G'\neq S_{n,m}$ is  a connected
threshold graph with degree sequence %$(d_{G'}(v_1), d_{G'}(v_2), \ldots, d_{G'}(v_n))$ and
$d_{G'}(v_1)\geq d_{G'}(v_2)\geq \ldots\geq d_{G'}(v_n)$.  Since there is only one  threshold graph $S_{n,m}$ in $\mathcal{H}_{n,m}$ for $n-1\le m\le n+1$, we have $m\ge n+2$. Further, since  $G'\neq S_{n,m}$ is a threshold graph,  it is easy to see that $d_{G'}(v_1)=n-1$, $n-2\ge d_{G'}(v_2)\ge d_{G'}(v_3)\ge 3$ and  $d(v_{n})=1$.  Let  $\delta_1=|\{i: d_{G'}(v_i)=1\}|$ and $\delta_2=|\{i: d_{G'}(v_i)=2, i>2 \}|$. Let $s$ be the largest positive number such that $d_{G'}(v_{s+2})\ge 3$ and $d_{G'}(v_{s+3})\le 2$. Denote by $\theta=d_{G'}(v_{s+2})-2>0$. Then the number $e(G')$ of $G'$ is at least $(n-1)+(n-\delta_1-2)+\theta+ (\theta-1)+\ldots+1=2n-3-\delta_1+\theta(\theta+1)/2.$ Hence   $2n-3-\delta_1+\theta(\theta+1)/2\le 2n-2$, which implies that $\delta_1 \ge \theta(\theta+1)/2-1$.  So $\delta_1\ge \theta$.  Let $G_1$ be obtained from $G'$ by a Transformation $(n-\delta_1+\theta,2;n-\delta_1-\delta_2,3;\theta,1)$ of Definition \ref{def}. By Lemma \ref{trg}, $\rho_{\alpha}(G_1) \ge \rho_{\alpha}(G')$. On the other hand, by the definition of $G'$ and $G_1\in  \mathcal{H}_{n,m}$,  $\rho_{\alpha}(G_1) \le \rho_{\alpha}(G')$. Hence  $\rho_{\alpha}(G_1) =\rho_{\alpha}(G')$. Hence by Lemma \ref{trg} again, we have $\alpha=\frac{1}{2}$, $\theta-1=0$, $n-\delta_1+\theta=2+3$, $n-\delta_1-\delta_2=4$, which implies  $\delta_2=0$ and  $\theta=1$, Hence $m=n+2$.   Furthermore, it is easy to see that $G'=\tilde{S}_{n,n+2}$.
So we finish our proof.
%if $\alpha \in (\frac{1}{2},1)$ or $\alpha=\frac{1}{2}$ and $m \neq n+2$, $\rho_{\alpha}(G_1) > \rho_{\alpha}(G')$, a contradiction to the choice of $G'$. Therefore, $G' = S_{n,m}$.
%If $m = n+2$ and $\alpha=\frac{1}{2}$, then $\rho_{\alpha}(G_1) = \rho_{\alpha}(G')$, i.e. $\rho_{\alpha}(\tilde{S}_{n, m})=\rho_{\alpha}(S_{n,m})$. Therefore, $S_{n,n+2}$ and $\tilde{S}_{n,n+2}$ are the extremal graphs that maximize the $A_{\frac{1}{2}}$-spectral radius in $\mathcal{H}_{n,n+2}$.
%Using the above result, If $\alpha \in (\frac{1}{2},1)$ or $\alpha=\frac{1}{2}$ and $m \neq n+2$, $S_{n,m}$ is the unique extremal graph maximizing $A_{\alpha}-$spectral radius, if $\alpha=\frac{1}{2}$ and $m = n+2$, there exist two extremal graph $\tilde{S}_{n, m}$ and $S_{n,m}$ that maximize $A_{\alpha-}$spectral radius.
%This completes the proof.
\end{proof}

%By Theorem \ref{t2n-2c}, we can easily prove that the graph $S_{n,2n-2}$ is the unique graph which maximizes the $Q$-spectral radius over $\mathcal{H}_{n,2n-2}$.\\

%by Theorem \ref{t2n-2c},  $G'\cong S_{n,2n-2}$.
%\begin{coro}\label{t2n-2}
%Let $n\ge 4$ and $m = 2n-2$. If  $G'$  is any graph which maximizes  $Q$-spectral radius of all graphs of order $n$ with size $m$, then $G'\cong K_5 \cup K_1$ for $n=6$; and $G'\cong S_{n,2n-2}$ for $n \neq 6$.
%\end{coro}
Now we are ready to present
{ \textbf{\mbox{Proof of Theorem \ref{t2n-2-1}}}}.
\begin{proof} It is easy to see that assertion holds for $n\le 16$ with the help of Python programming. So we assume that $n > 16$.
Let $G'$ be any graph which maximizes the spectral signless radius of all graphs of order $n$ with size $m=2n-2$.
  Then  $G'$ is a threshold graph. Suppose that  $G'$ is disconnected. Note that any threshold graph has at most one non-trivial component. So we assume that $G'= G^\ast \bigcup \bar{K}_{s}$, where $G^*$ is a connected threshold graph of order $n^*$ with size $m^*=2n-2$. Then by $n^*(n^*-1)/2\ge m^*=2n-2$, we have
   $\frac{1+\sqrt{16n-15}}{2}\le n^*\le n-1$   and  $q(G')=q(G^\ast)$.
%First we prove the following result holds: $q(G^\ast) \le n+1.5$.
By Lemma \ref{ue},
$$q(G^*) \le \frac{2(2n-2)}{n^*-1}+n^*-2.$$
Let $f(x)=\frac{4n-4}{x}+x-1$ for $x \in [ \frac{-1+\sqrt{16n-15}}{2},n-2]$.  By $f^{\prime\prime}(x)>0$,  we have
  $f(x)\le \max\{ f(\frac{-1+\sqrt{16n-15}}{2}),  f(n-2)\}.$ Hence by   $\frac{-1+\sqrt{16n-15}}{2}\le n^*-1\le n-2$, we have
   $$q(G^*) \le \max\{\sqrt{16n-15}-1, n+\frac{4}{n-2}\}< n+1.6\le q(S_{n,2n-2}),$$
which is a contradiction. Hence $G'$ must be a connected graph. Therefore, by Theorem \ref{t2n-2c},  $G'= S_{n,2n-2}$.
 %  occurs if and only if  either $x=n-2$ or $x=\frac{1+\sqrt{16n-15}}{2}$. A simple calculation yields that $f(n-2)=n+1+\frac{4}{n-2}$ and
%$f(\frac{1+\sqrt{16n-15}}{2})=\sqrt{16n-15}-1$.
%Since $n > 16$, then $f(1)<n+1.5$ and $f(n-2\sqrt{n-1})<n+1.5$. Hence $q(G^\ast) \le n+1.5$.
%With the help of Python programming, when $7 \le n \le 13$,  $\sigma(G^\ast) \le n+1.5$.
%On the other hand, by Lemma \ref{2n-2}, $q(S_{n,2n-2}) \ge n+1.6>q(G')$,  which  is a contradiction. Hence, $G'$ is connected. Moreover, by Theorem %\ref{t2n-2c},  $G'\cong S_{n,2n-2}$.
\end{proof}

%\begin{rem}
%By Corollary \ref{t2n-2}, $S_{n,2n-2}$ is the unique graph  which  maximizes the $Q$-spectral radius of all graphs of order $n$ and size $m=2n-2$ except for $n=6$,  which solves Problem \ref{que}.
%\end{rem}

In order to prove Theorem~\ref{tknc}, we firstly prove the following key theorem.
\begin{theo}\label{tknc-1}
Let $r\ge3$, $n$ and $m$  be three positive number with $n >\frac{30r-63+5\sqrt{32r^2-136r+137}}{2}$ and $(r-1)n-\frac{r(r-1)}{2}< m \le rn-\frac{r(r+1)}{2}$. \\

 (i). If $\alpha \in (\frac{1}{2}, 1)$ or $\alpha=\frac{1}{2}$ and $m \neq (r-1)n-\frac{r(r-1)}{2}+3$, then $S_{n, m}$ is the unique  extremal graph that maximizes the $A_{\alpha}$-spectral radius in $\mathcal{H}_{n,m}$. \\

  (ii). If $\alpha=\frac{1}{2}$ and $m = (r-1)n-\frac{r(r-1)}{2}+3$, then $S_{n, m}$ and $\tilde{S}_{n, m}$ are the only two extremal graphs that maximize the $A_{\frac{1}{2}}$-spectral radius in $\mathcal{H}_{n,m}$.
\end{theo}

{\it{Proof}}.
Note that $\rho_{\alpha}(S_{n, m})=\rho_{\alpha}(\tilde{S}_{n, m})$ for $m = (r-1)n-\frac{r(r-1)}{2}+3$. Hence it is sufficient to prove that if
$G'\neq S_{n.m}$  is  any graph which has the maximal $A_\alpha$-spectral radius of  $ \mathcal{H}_{n,m}$, then $\alpha=\frac{1}{2}$, $m = (r-1)n-\frac{r(r-1)}{2}+3$ and $G'=\tilde{S}_{n, m}$.
%\end{proof}

%In order to prove Theorem~\ref{tknc} for $r\ge 3$, we  need to prove some properties of  extremal graphs  which attain the maximal $A_\alpha$-spectral radius of  $ \mathcal{H}$ which is the set of  all connected graphs of order $n$ with size $m$.

Let $G'\neq S_{n,m}$  be any graph which has the maximal $A_\alpha$-spectral radius of  $ \mathcal{H}_{n,m}$. By Lemma~\ref{threshold}, we assume that $G'$  is a connected threshold graph with the $n\times n$ stepwise adjacency matrix $A(G')=(a_{ij})_{n\times n}$ and the degree sequence
 %$(d_{G'}(v_1), d_{G'}(v_2), \ldots, d_{G'}(v_n))$, where
 $d_{G'}(v_1)\geq d_{G'}(v_2)\geq \ldots\geq d_{G'}(v_n)$.  Moreover,  denote by
 $$\kappa=\kappa(G')=\max\{j:a_{j+1,j}=1, 1\le j\le n-1 \}$$ and
 $$\delta_j=\delta_j(G')=|\{i: d_{G'}(v_i)=j, n\ge i>j\}|$$
  for $j=1,\cdots,\kappa$. Then we will prove a series of lemmas.
% \end{proof}

\begin{lemma}\label{lemma11}
%Let $G'\neq S_{n,m}$  be a  graph having the maximal $A_\alpha$-spectral radius of  $ \mathcal{H}_{n,m}$.  If $(r-1)n-\frac{r(r-1)}{2}<m\le rn-\frac{r(r+1)}{2}$, then
  $d_{G'}(v_n)\le \kappa(G')-2$, $d_{G'}(v_{r+1})\ge  r+1$ and $\kappa(G')\ge r+1$.
\end{lemma}

\begin{proof}
%If $r=1$, then $m\le n-1$, there exists only one connected threshold graph $G'=S_{n,m}$, a contradiction to $G'\neq S_{n,m}$, so the result holds.
%If $r=2$, by Theorem \ref{t2n-2c}, we have $G'=\widetilde{S}_{n,m}$, then $d_{G'}(v_{3})=3$ and $\kappa(G')=3$, the result holds. In the following we
%only consider $r \ge 3$.

 By the definition of $\kappa(G')$,  we have $a_{\kappa+2, \kappa+1}=0$. By  $G'\neq S_{n,m}$ being a connected threshold graph,  we have $a_{n, \kappa+1}=0$ which implies that  $d_{G'}(v_n)\le \kappa$. Further, if $d_{G'}(v_n)= \kappa$, then the degree sequence of $G'$ is $(n-1, \ldots, n-1, \kappa, \ldots, \kappa)$  which implies $G'=K_{\kappa}\bigvee (n-\kappa)K_1=S_{n,m}$, a contradiction. If $d_{G'}(v_n)= \kappa-1$,  then
      the degree sequence of $G'$ is $(n-1, \ldots, n-1, \kappa, \ldots, \kappa, \kappa-1, \ldots, \kappa-1))$ which implies that
$G'=K_{\kappa-1}\bigvee (K_{1,\delta_\kappa}\bigcup(n-\kappa-\delta_\kappa)K_1)=S_{n,m}$, a contradiction. Hence  $d_{G'}(v_n)\le \kappa(G')-2$.

Suppose that  $d_{G'}(v_{r+1})\le r-1$.  Since $A(G')=(a_{ij})$ is a stepwise adjacency matrix of $G'$ with the degree sequence  $d_{G'}(v_1)\geq d_{G'}(v_2)\geq \ldots\geq d_{G'}(v_n)$,  we have $a_{r+1,r}=0$ and  $a_{r,r+1}=a_{r+1,r}=0$. So  $d_{G'}(v_r)\le r-1$,  which implies  $\sum_{j=r}^n d_{G'}(v_j)\le(r-1)(n-r+1)$. Hence,
\begin{equation*}
 \begin{array}{lll}
2m&=&\sum_{j=1}^n d_{G'}(v_j)\\
&=&\sum_{j=1}^{r-1} d_{G'}(v_j)+\sum_{j=r}^{n} d_{G'}(v_j)\\
&\le&(n-1)(r-1)+(r-1)(n-r+1)\\
&=&2(r-1)n-(r-1)r.
\end{array}
\end{equation*}
So $m\le (r-1)n-\frac{r(r-1)}{2}$, which contradicts to the assumption of $m$.
Hence $d_{G'}(v_{r+1})\ge r$.

Furthermore, suppose that $d_{G'}(v_{r+1})= r$.
  By $r+1 \in \{i: d_{G'}(v_i)=r, i>r\} $,  we have $\delta_{r}>0$ and $d_{G'}(v_{n})\le \cdots \le d_{G'}(v_{r+2})\le r$.  In addition, by
  $m\le rn-\frac{r(r+1)}{2}$, we have $d_{G'}(v_{n})<r$.  Let $\varsigma=\max\{j:\delta_j \neq 0,j<r \}$. We consider the following two cases:

{\bf Case 1: $\varsigma<r-1$. }

%If $\varsigma<r-1$, by $(r-1)n-\frac{r(r-1)}{2}< m $, we have
%$r-\varsigma-1 \le \delta_{r}$,
Since $A(G')$ is the stepwise adjacency  symmetric matrix, we  have $a_{ij}=0, \mbox{ for all } i>j>r$. Hence
\begin{equation*}
 \begin{array}{lll}
 (r-1)n-\frac{r(r-1)}{2} &<& m\\
&=&\sum_{j=1}^n \sum_{i>j} a_{ij}\\
&=&\sum_{j=1}^r \sum_{i>j} a_{ij}+\sum_{j=r+1}^n \sum_{i>j} a_{ij}\\
&=&\sum_{j=1}^r \sum_{i>j} a_{ij}\\
&=&\sum_{j=1}^r \sum_{i=1}^n a_{ij}-\sum_{j=1}^r \sum_{i\le j} a_{ij}\\
&=&\sum_{j=1}^{\varsigma} \sum_{i=1}^n a_{ij}+\sum_{j=\varsigma+1}^{r-1} \sum_{i=1}^n a_{ij}+ \sum_{i=1}^n a_{ir}-\frac{r(r-1)}{2}\\
&\le&\varsigma(n-1)+(r-\varsigma-1)(n-2)+\delta_r+r-1-\frac{r(r-1)}{2}\\
&=&(r-1)n-\frac{r(r-1)}{2}+\delta_r-(r-1)+\varsigma.
\end{array}
\end{equation*}
So $r-\varsigma-1 <\delta_{r}$.
Let $G_1$ be obtained from $G'$ by a Transformation $(n-\sum_{j=1}^{\varsigma}\delta_j+1,r-1;n-\sum_{j=1}^{\varsigma}\delta_j,r;1,r-\varsigma-1)$ of Definition \ref{def2}. By Lemma \ref{trg1}, $\rho_{\alpha}(G_1) \ge \rho_{\alpha}(G')$.  Hence  $\rho_{\alpha}(G_1) =\rho_{\alpha}(G')$.
           By Lemma \ref{trg1} again, we have  $n-\sum_{j=1}^{\varsigma}\delta_j=r+1$, $\alpha=1/2$, $r=\varsigma+2$. Hence $m\le(r-1)n-\frac{r(r-1)}{2}$ which contradicts to the assumption
$m>(r-1)n-\frac{r(r-1)}{2}$.  It is impossible.
%Therefore, $\rho_{\alpha}(G_1) > \rho_{\alpha}(G')$, a contradiction to the choice of $G'$. \\

{\bf Case 2: $\varsigma=r-1$. }

 Since $G'\neq S_{n,m}$, it is easy to see that  $d_{G'}(v_{n})<r-1$ (otherwise, $G'=K_{r-1}\vee (K_{1,\delta_{r}} \cup (n-r-\delta_{r})K_1)$, so $G'=S_{n,m}$).   Let $\varsigma_1=\max\{j:\delta_j \neq 0,j<r-1\}$. Clearly,
 \begin{equation*}
 \begin{array}{lll}
 (r-1)n-\frac{r(r-1)}{2} &<& m\\
&=&\sum_{j=1}^n \sum_{i>j} a_{ij}\\
&=&\sum_{j=1}^r \sum_{i=1}^n a_{ij}-\sum_{j=1}^r \sum_{i\le j} a_{ij}\\
&=&\sum_{j=1}^{\varsigma_1} \sum_{i=1}^n a_{ij}+\sum_{j=\varsigma_1+1}^{r-1} \sum_{i=1}^n a_{ij}+ \sum_{i=1}^n a_{ir}-\frac{r(r-1)}{2}\\
&\le&\varsigma_1(n-1)+(r-\varsigma_1-1)(n-2)+\delta_r+r-1-\frac{r(r-1)}{2}\\
&=&(r-1)n-\frac{r(r-1)}{2}+\delta_r-(r-1)+\varsigma_1.
\end{array}
\end{equation*}
Then   $r-\varsigma_1-1< \delta_r$. Let $G_1$ be obtained from $G'$ by a Transformation $(n-\sum_{j=1}^{\varsigma_1}\delta_j+1,r-1;n-\sum_{j=1}^{r-1}\delta_j,r;1,r-\varsigma_1-1)$ of Definition \ref{def2}. By $n-\sum_{j=1}^{\varsigma_1}\delta_j+1-n+\sum_{j=1}^{r-1}\delta_j=\delta_{r-1}+1\ge 2$ and  Lemma \ref{trg1}, $\rho_{\alpha}(G_1) > \rho_{\alpha}(G')$,  which is a contradiction to the choice of $G'$. It is impossible.
Hence $d_{G'}(v_{r+1})\ge r+1$. Then we have $a_{r+2,r+1}=a_{r+1,r+2}=1$, so $\kappa \ge r+1$.
This completes the proof.
\end{proof}

\begin{lemma}\label{lemma12}
%Let $G'\neq S_{n,m}$  be a  graph having the maximal $A_\alpha$-spectral radius of  $ \mathcal{H}_{n,m}$.
%Denote by  $\kappa=\kappa(G')=\max\{j:a_{j+1,j}=1 \}\ge 3$ and $\delta_j=\delta_j(G')=|\{i: d_{G'}(v_i)=j, i>j\}|$ for $j=1,\cdots,\kappa$.
 %Then
  $\delta_{j}\times\delta_{j+1}\times\delta_{j+2}=0$, for $j=1, 2, , \ldots, \kappa-2$.
\end{lemma}

\begin{proof}
Suppose there exists an $1\leq h\leq \kappa-2$ such that $\delta_{h}\times\delta_{h+1}\times\delta_{h+2}\neq 0$. Let $G_1$ be the threshold graph
 obtained from $G'$ by a Transformation $(n-\sum_{j=1}^{h}\delta_j+1,h;n-\sum_{j=1}^{h+1}\delta_j,h+1;1,1)$ of Definition \ref{def}. Since $n-\sum_{j=1}^{h}\delta_j+1-(n-\sum_{j=1}^{h+1}\delta_j)=\delta_{h+1}+1\ge 2$,  we have  $\rho_{\alpha}(G_1) > \rho_{\alpha}(G')$ by Lemma \ref{trg}, which contradicts to the choice of $G'$. So the assertion holds.
\end{proof}

\begin{lemma}\label{lemma16}
  Let $s$ be the largest positive number such that $d(v_{r+s})\ge r+1$ and  $d(v_{r+s+1})\le r$ with  $\theta=d(v_{r+s})-r$. Then  $ \kappa\ge r+1$, $\delta_{\kappa}>0$,  $r+s\ge \kappa+1$ and  $s\ge \theta+1$.
 \end{lemma}

 \begin{proof}
 %By the definition of $\kappa$, $a_{\kappa+2, \kappa+1}=0$ which implies $a_{\kappa+1, \kappa+2}=0$. In addition, $a_{\kappa+1, \kappa}=1$. So
% $d(v_{\kappa+1})=\kappa$ and $\delta_\kappa>0$.
%Furthermore, $d(v_{i})\le \kappa$ for $i=\kappa+2, \ldots, n$. Hence  $(r-1)n-\frac{r(r-1)}{2}<m\le \kappa(n-\kappa)+\frac{\kappa(\kappa-1)}{2}=\kappa n-\frac{\kappa(\kappa+1)}{2}$, which implies $\kappa> r+1$.

 By the definition of $s$, $a_{r+s+1,r+s}=0$. Hence by the definition of $\kappa$, we have  $a_{\kappa+1, \kappa}=1$ and $r+s\ge \kappa+1$.

 Suppose that $s\le \theta$.  By the definition of the stepwise adjacency matrix $A(G')$, we have $a_{r+s,r+\theta+1}=1$ and $a_{r+s,r+\theta+2}=0$. By the definition of $s$, we have $a_{r+s+1,r+\theta+1}=0$. In addition, since $A(G')$ is symmetric, we have $a_{r+\theta+1,r+s}=1$ and $a_{r+\theta+1,r+s+1}=0$, which implies that
$d_{G'}(v_{r+\theta+1})=r+s< r+\theta$, it is a contradiction. So the assertion holds.
 \end{proof}

\begin{lemma}\label{lemma17}
The edge number of the induced subgraph by vertex set $U_2=\{v_{r+1}, \ldots, v_n\}$ is at most $\sum_{j=1}^{r-1}(r-j)\delta_j$.
\end{lemma}
\begin{proof}
Let $U_1=\{v_1, \ldots, v_r\}$ and $U_2=\{v_{r+1}, \ldots, v_n\}$.  Denote by $e(U_1,U_2)$ the edge number  between vertex sets $U_1$ and $U_2$, and $e(U_2)$ the edge number in vertex set $U_2$, respectively.
 %Denote by
%$\Phi_0=r(n-r)-e(U_1, U_2)$ and $\Phi_1=e(U_2)$,respectively.   Then
%$\Phi_0=\sum_{j=1}^{r-1}(r-j)\delta_j$ and $\Phi_0\ge \Phi_1$.
 Since $G'\neq S_{n,m}$ is a connected threshold graph with $m> (r-1)n-\frac{r(r-1)}{2}$,   $d(v_{r+1})\ge r$ and   $U_1$ is a clique of order $r$ by Lemma~\ref{lemma11}. Let $W_j=\{u\in U_2:\ d(u)=j\}$ for $j=1, \ldots, r-1.$
 By the definition of $\delta_j$,  we have $|W_j|=\delta_j$. Further, by $G'$ being a threshold graph, each vertex in $W_j$ is not adjacent to $v_{j+1}, \ldots, v_r$  for $j=1, \ldots, r-1$ and each vertex in $U_2\backslash \bigcup_{j=1}^{r-1}W_j$ is adjacent to $v_1, \ldots, v_r$. Hence
 $$e(U_1, U_2)=r(n-r)-\sum_{i=1}^{r-1}(r-j)\delta_j.$$
   Furthermore, by $m\le rn-\frac{r(r+1)}{2}$, we have
 $$e(U_1)+e(U_1, U_2)+e(U_2)=\frac{r(r-1)}{2}+r(n-r)-\sum_{i=1}^{r-1}(r-j)\delta_j+e(U_2)=m\le rn-\frac{r(r+1)}{2},$$  which implies $e(U_2)\le \sum_{j=1}^{r-1}(r-j)\delta_j$.
\end{proof}

\begin{lemma}\label{lemma13}
%Let $G'\neq S_{n,m}$  be a  graph having the maximal $A_\alpha$-spectral radius of  $ \mathcal{H}_{n,m}$.
%Denote by $\kappa=\kappa(G')=\max\{j:a_{j+1,j}=1 \}$ and $\delta_j=\delta_j(G')=|\{i: d_{G'}(v_i)=j, i>j\}|$ for $j=1,\cdots,\kappa$.
If there exist two positive integers $1\leq h\leq \kappa-4$ and $3\leq l\leq \kappa-h-1$
such that $\delta_h> 0$, $\delta_{h+1}>0$, $\delta_{h+2}=\cdots=\delta_{h+l}=0$, $\delta_{h+l+1}> 0$, which is depicted in Figure \ref{fig:fhl}, then

(i).  $\delta_h \le l-1$ and  $\delta_{h+1} \le l-2$. %, $\sum_{j=h}^{h+l} \delta_j \le 2l-3 $;

(ii).  $\delta_{h+l+1} \le l-2$ and  $\delta_{h+l+2} \le l-1$ if $h+l+2\le \kappa$.  %, $\sum_{j=h+l+1}^{h+l+2} \delta_j \le 2(l+1)-5$, if $\delta_{h+l+2}$ exists.

(iii). $\sum_{j=h}^{h+l} \delta_j \le 2l-3 $\  and\ $\sum_{j=h+l+1}^{h+l+2} \delta_j \le 2(l+1)-5$\ if $h+l+2\le \kappa$.

\begin{figure}[htbp]
\centering
\begin{tabular}{cccccccc}
  & & & &\multicolumn{4}{c}{\textbf{\begin{tabular}[c]{@{}c@{}}$l$\end{tabular}}} \\
  & & & &\multicolumn{4}{c}{\downbracefill}  \\
 \hline
 & & h &h+1          & h+2 &  $\cdots$  &  h+$l$ & h+$l$+1  \\
 \hline
     & $n-\sum_{j=1}^{h+1}\delta_j$
      & 1 &1          &  1 &  $\cdots$   & 1   &   1 \\
 %\multirow{4}{*}{\emph{$\delta_{h+1}$}\Bigg\{ }
  & $n-\sum_{j=1}^{h+1}\delta_j+1$  & 1 &1          &   &     &    &    \\
  &$\vdots$ & $\vdots$ &$\vdots$          &   &     &    &    \\
  & $n-\sum_{j=1}^{h}\delta_j$   & 1 &1          &   &     &    &    \\
%\multirow{4}{*}{\emph{\mbox{$\delta_{h}$  } }\Bigg\{ }
   & $n-\sum_{j=1}^{h}\delta_j+1$  & 1 &0          &   &     &    &    \\
 & $\vdots$ & $\vdots$ &$\vdots$         &   &     &    &    \\
 & $n-\sum_{j=1}^{h-1}\delta_j$  & 1 &   0      &   &     &    &    \\
\end{tabular}%
\caption{$\delta_{h}\neq 0$, $\delta_{h+1}\neq 0$, $\delta_{h+l+1}\neq 0$, $\delta_{h+2}=\cdots=\delta_{h+l}=0$}
\label{fig:fhl}
\end{figure}
%By Claim \ref{h0}, $l>1$.
\end{lemma}

\begin{proof} (i).   Suppose  that $\delta_h \ge l$. Let  $G_1$ be obtained from $G'$ by a Transformation $(n-\sum_{j=1}^{h}\delta_j+l,h+1;n-\sum_{j=1}^{h+1}\delta_j,h+2;l,1)$ of Definition \ref{def}.  Clearly $G_1\in \mathcal{H}_{n,m}$. By Lemma \ref{trg}, $\rho_{\alpha}(G_1) > \rho_{\alpha}(G')$, which contradicts to $G'$ having the maximal $A_{\alpha}$-spectral radius in  $\mathcal{H}_{n,m}$. Hence $\delta_h \leq l-1$.

Furthermore, suppose that $\delta_{h+1} \ge l-1$. Let   $G_2$  be the threshold graph of order $n$ with size $m$ obtained from $G'$ by a Transformation $(n-\sum_{j=1}^{h+1}\delta_j+l-1,h+2;n-\sum_{j=1}^{h+1}\delta_j,h+3;l-1,1)$ of Definition \ref{def} and $G_3$  be the graph obtained from $G_2$ by a Transformation $(n-\sum_{j=1}^{h}\delta_j+1,h+1;n-\sum_{j=1}^{h+1}\delta_j+l-1,h+2;1,1)$ of Definition \ref{def}. By Lemma \ref{trg}, $\rho_{\alpha}(G_2) \geq \rho_{\alpha}(G')$ and $\rho_{\alpha}(G_3) \geq \rho_{\alpha}(G_2)$.  On the other hand, by $G_3\in \mathcal{H}_{n,m}$, we have $\rho_{\alpha}(G_3)\le \rho_{\alpha}(G')$. Hence $ \rho_{\alpha}(G')=\rho_{\alpha}(G_2) =\rho_{\alpha}(G_3) $. By Lemma \ref{trg} again,  we have $n-\sum_{j=1}^{h+1}\delta_j+l-1=h+5$ and $n-\sum_{j=1}^{h+1}\delta_j+l-1=h+3$, which is a contradiction. Therefore,  $\delta_{h+1} \leq l-2$.

%Since $\delta_h \le l-1$, $\delta_{h+1} \le l-2$, $\delta_{h+2}=\cdots=\delta_{h+l}=0$, we have $\sum_{j=h}^{h+l} \delta_j \le 2(l+1)-5$.

{(ii).} Suppose that  $\delta_{h+l+1} \ge l-1$. Let  $G_1$ be the threshold graph obtained from $G'$ by a Transformation $(n-\sum_{j=1}^{h+1}\delta_j+l,h+l;n-\sum_{j=1}^{h+1}\delta_j,h+l+1;1,l-1)$ of Definition \ref{def2}.  By Lemma \ref{trg1}, $\rho_{\alpha}(G_1) > \rho_{\alpha}(G')$, which contradicts  to $G'$ having the maximal $A_{\alpha}$-spectral radius in $\mathcal{H}_{n,m}$. Hence $\delta_{h+l+1} \leq l-2$.

Further, assume that $h+l+2\le \kappa$. If $\delta_{h+l+2} \ge l$, let $G_2$  be the graph obtained from $G'$ by a Transformation $(n-\sum_{j=1}^{h+1}\delta_j+1,h+l+1;n-\sum_{j=1}^{h+l+1}\delta_j,h+l+2;1,l)$ of Definition \ref{def2}. By Lemma \ref{trg1}, $\rho_{\alpha}(G_2) > \rho_{\alpha}(G')$, which contradicts to $G'$ having the maximal $A_{\alpha}$-spectral radius in $\mathcal{H}_{n,m}$. Hence  $\delta_{h+l+2} \leq l-1$.

(iii). It is easy to see that (iii) follows from (i) and (ii).
\end{proof}
%\end {document}

\begin{lemma}\label{lemma14}
%Let $G'\neq S_{n,m}$  be a  graph having the maximal $A_\alpha$-spectral radius of  $ \mathcal{H}_{n,m}$.
%Denote by $\kappa=\kappa(G')=\max\{j:a_{j+1,j}=1 \}$ and $\delta_j=\delta_j(G')=|\{i: d_{G'}(v_i)=j, i>j\}|$ for $j=1,\cdots,\kappa$.
If there exist two positive integers $1\leq h\le \kappa-3$ and $2\leq l\leq \kappa-h-1$ such that $\delta_h>0$, $\delta_{h+1}=\delta_{h+2}=\cdots=\delta_{h+l}=0$, $\delta_{h+l+1}>0$, then

(i). $\delta_h \le l-1$ % $\sum_{j=h}^{h+l} \delta_j < 2(l+1)-5$,
 and $\delta_{h+l+1} \le l-1$.

(ii). If $h+l+2\le \kappa$, then  $\delta_{h+l+2} \le l$.
\end{lemma}
\begin{proof}
%The proof of $\delta_h \le l-1$ is similar to Claim \ref{h1}.
(i). Suppose that $\delta_h \ge l$. Let $G_1$ be obtained from $G'$ by a Transformation $(n-\sum_{j=1}^{h}\delta_j+l,h+1;n-\sum_{j=1}^{h}\delta_j,h+2;l,1)$ of Definition \ref{def}. By Lemma \ref{trg} and $l\ge 2>1$, $\rho_{\alpha}(G_1) > \rho_{\alpha}(G')$, which contradicts to $G'$ having the maximal $A_{\alpha}$-spectral radius in $\mathcal{H}_{n,m}$. Hence, $\delta_h \le l-1$.

Suppose that $\delta_{h+l+1} \ge l$.  Let  $G_2$ be obtained from $G'$ by a Transformation $(n-\sum_{j=1}^{h}\delta_j+l,h+l;n-\sum_{j=1}^{h}\delta_j,h+l+1;1,l)$ of Definition \ref{def2}.  By Lemma \ref{trg1} and $l\ge 2>1$, $\rho_{\alpha}(G_1) > \rho_{\alpha}(G')$, which contradicts  to $G'$ having the maximal $A_{\alpha}$-spectral radius in $\mathcal{H}_{n,m}$.  Hence $\delta_{h+l+1} \leq l-1$.
 Therefore (i) holds.

 (ii). Suppose  that $\delta_{h+l+2} \ge l+1$. Let $G_2$  be the graph obtained from $G'$ by a Transformation $(n-\sum_{j=1}^{h}\delta_j+1,h+l+1;n-\sum_{j=1}^{h+l+1}\delta_j,h+l+2;1,l+1)$ of Definition \ref{def2}. By Lemma \ref{trg1} and $l+1\ge 3>1$, we have $\rho_{\alpha}(G_2) > \rho_{\alpha}(G')$, which contradicts to $G'$ having the maximal $A_{\alpha}$-spectral radius in $\mathcal{H}_{n,m}$.  Hence $\delta_{h+l+2} \leq l$.  So (ii) holds.
 \end{proof}

\begin{lemma}\label{lemma15}
%Let $G'\neq S_{n,m}$  be a  graph having the maximal $A_\alpha$-spectral radius of  $ \mathcal{H}_{n,m}$.
%Denote by $\kappa=\kappa(G')=\max\{j:a_{j+1,j}=1 \}$ and $\delta_j=\delta_j(G')=|\{i: d_{G'}(v_i)=j, i>j\}|$ for $j=1,\cdots,\kappa$.
If there exists one positive integer $1\leq h\leq \kappa-2$ such that $\delta_h> 0$, $\delta_{h+1}=0$ and $\delta_{h+2}> 0$, then  $\alpha =\frac{1}{2}$, $h=\kappa-2$,  $\delta_{\kappa}=1$ and $n-\sum_{j=1}^{\kappa-2}\delta_j=\kappa+1$.
\end{lemma}

\begin{proof} Let $G_1$  be the connected threshold graph  obtained from $G'$ by a Transformation $(n-\sum_{j=1}^{h}\delta_j+1,h+1;n-\sum_{j=1}^{h}\delta_j,h+2)$ of Definition \ref{def0}. By  Lemma \ref{trij}, $\rho_{\alpha}(G_1) \ge \rho_{\alpha}(G')$. On the other hand, by the definition of $G'$, we have  $\rho_{\alpha}(G_1) \le \rho_{\alpha}(G')$. Hence  $\rho_{\alpha}(G_1) =\rho_{\alpha}(G')$. Therefore  by  Lemma \ref{trij} again, we have $\alpha =\frac{1}{2}$  and  $n-\sum_{j=1}^{h}\delta_j=h+3$.
In addition, by the definition of $\kappa$, it is easy to see that $n=\sum_{j=1}^\kappa\delta_j+\kappa$. Hence by $\delta_{\kappa}\ge 1$,
$$n= \sum_{j=1}^\kappa\delta_j+\kappa=n+( \sum_{j=h+1}^\kappa\delta_j-1)+(\kappa-h-2)\ge n.$$
Therefore   $\kappa-h-2=0$ and $\sum_{j=h+1}^\kappa\delta_j=1$. So the assertion holds.
\end{proof}

 In order to present more properties of extremal graphs in $\mathcal{H}_{n,m}$, we also need the following symbol.
Let $G'\neq S_{n,m}$ be a connected threshold graph with the $n\times n$ stepwise adjacency matrix $A(G')=(a_{ij})$ and the degree sequence  $d_{G'}(v_1)\geq d_{G'}(v_2)\geq \ldots\geq d_{G'}(v_n)$. Let $s$ be the largest positive number such that $d(v_{r+s})\ge r+1$ and  $d(v_{r+s+1})\le r$. Moreover denote by $\theta=d(v_{r+s})-r$.
%Furthermore, let $n_0=n$,  $h_0=d_{G'}(v_{n_0})<\kappa-1$, $n_i=n_{i-1}-\delta_{h_{i-1}}-\delta_{h_{i-1}+1}$, $h_i=d_{G'}(v_{n_i})<\kappa-1$, $l_i=h_i-h_{i-1}-1\ge 1$,  where $i=1, \ldots, t-1$, and $n_t= n_{t-1}-\delta_{h_{t-1}}-\delta_{h_{t-1}+1}$, $h_t=d_{G'}(v_{n_t})\ge \kappa-1$.

%Let $G'\neq S_{n,m}$ be a connected threshold graph with the $n\times n$ stepwise adjacency matrix $A(G')=(a_{ij})$ and the degree sequence  $d_{G'}(v_1)\geq d_{G'}(v_2)\geq \ldots\geq d_{G'}(v_n)$.  Then
% it is easy to see that $d_{G'}(v_n)\le \kappa-2$ (otherwise  $d_{G'}(v_n)\ge \kappa-1$  which implies that $m\ge rn-\frac{r(r+1)}{2}$ by $\kappa\ge r+1$. Hence $m\ge rn-\frac{r(r+1)}{2}$ and $G=S_{n,m}$).
 %it is easy to see that $d_{G'}(v_n)\le \kappa-2$ (If $d_{G'}(v_n)= \kappa$, then the degree sequence of $G'$ is $(n-1, \ldots, n-1, \kappa, \ldots \kappa)$  which implies $G'=K_{\kappa}\vee (n-\kappa)K_1=S_{n,m}$. If $d_{G'}(v_n)= \kappa-1$, then  the degree sequence of $G'$ is $(n-1, \ldots, n-1, \kappa, \ldots \kappa, \kappa-1, \ldots, \kappa-1))$ which implies that
%$G'=K_{\kappa-1}\vee (K_{1,\delta_\kappa}\cup(n-\kappa-\delta_\kappa)K_1)=S_{n,m}$).
 %Therefore,  let $n_0=n$,  $h_0=d_{G'}(v_{n_0})\le \kappa-2$, $n_i=n_{i-1}-\delta_{h_{i-1}}-\delta_{h_{i-1}+1}$, $h_i=d_{G'}(v_{n_i})\le \kappa-2$, $l_i=h_i-h_{i-1}-1\ge 1$,  where $i=1, \ldots, t-1$, and $n_t= n_{t-1}-\delta_{h_{t-1}}-\delta_{h_{t-1}+1}$, $h_t=d_{G'}(v_{n_t})\ge \kappa-1$.

\begin{lemma}\label{lemma18}
There exists a $1\le h\le \kappa-2$ such that $\delta_h>0$, $\delta_{h+1}=0$ and $\delta_{h+2}>0$.
\end{lemma}

\begin{proof}

By Lemma \ref{lemma11}, we $d_{G'}(v_n)\le \kappa-2$. We define a sequence of numbers by recursive methods.

 $$n_0=n, \ \ h_0=d_{G'}(v_{n_0})\le \kappa-2,\ \ l_0=0;$$
  $$n_i=n_{i-1}-\delta_{h_{i-1}}-\delta_{h_{i-1}+1},\ \ h_i=d_{G'}(v_{n_i})\le \kappa-2, \ \ l_i=h_i-h_{i-1}-1\ge 1,$$
    for $i=1, \ldots, t-1$; and
    $$n_t= n_{t-1}-\delta_{h_{t-1}}-\delta_{h_{t-1}+1},\ \ h_t=d_{G'}(v_{n_t})\ge \kappa-1.$$

%Since $G'\neq S_{n,m}$ be a connected threshold graph, .
Suppose that there does not exist any  $1\le h\le \kappa-2$ such that $\delta_h>0$, $\delta_{h+1}=0$ and $\delta_{h+2}>0$.

Since $d_{G'}(v_n)\le \kappa-2$, $h_0=d_{G'}(v_{n_0})\le \kappa-2$ and $h_t=d_{G'}(v_{n_t})\ge \kappa-1$,
we have  $t\ge 1$ and $h_t \le \kappa$. Moreover,
\begin{equation}\label{lemma18-0}
n_t\le \kappa+\delta_\kappa+\delta_{\kappa-1}.
\end{equation}
Further by Lemmas \ref{lemma13} and \ref{lemma14},  we have
  \begin{equation}\label{lemma18-1}
  n_{i-1}-n_i=\sum_{j=h_{i-1}}^{h_i-1}\delta_j\le 2(h_i-h_{i-1})-5, \ \ \ \ \ i=1, \ldots, t.
  \end{equation}
  Hence by $h_t\ge\kappa-1$, $h_0\ge 1$, $t\ge 1$ and (\ref{lemma18-1}), we have
  \begin{equation}\label{lemma18-2}
  n_0-n_t=\sum_{i=1}^t(n_{i-1}-n_i)\le \sum_{i=1}^t[2(h_i-h_{i-1})-5]=2(h_t-h_0)-5t\le 2\kappa-7.
  \end{equation}
In addition,  by $h_t \ge \kappa-1$, (\ref{lemma18-2}) and the definition of $h_i$, we have
\begin{equation}  \label{lemma18-3}
\sum_{i=1}^{\kappa-2}\delta_i\le \sum_{i=1}^{h_t-1}\delta_i=\sum_{i=1}^t(\delta_{h_{i-1}}+\delta_{h_{i-1}+1})=\sum_{i=1}^t(n_{i-1}-n_i)=n_0-n_t\le 2\kappa-7.
\end{equation}
  In particular, by $h_1 \le \kappa$,  Lemmas ~\ref{lemma13} and \ref{lemma14}, we have
  \begin{equation}\label{lemma18-31}
  \delta_{h_0}\le h_1-h_0-1-1\le \kappa-3, \ \ \ \delta_{h_0+1}\le (h_1-h_0-1)-2 \le \kappa-4.
  \end{equation}

 %Let $\Phi_0$ and $\Phi_1$ be definition in Lemma \ref{lemma17}. Then
By (\ref{lemma18-3}), (\ref{lemma18-31})  and $\kappa\ge r+1$, we have

\begin{eqnarray*}
%\Phi_0 &=& \sum_{i=1}^{r-1}(r-i)\delta_i\\
 \sum_{i=1}^{r-1}(r-i)\delta_i
&\le &  (r-1)\delta_1+(r-2)\delta_2+(r-3)\sum_{i=1}^{r-1}\delta_i\\
&= & 2\delta_1+\delta_2+(r-3)\sum_{i=3}^{r-1}\delta_i\\
&\le & 2\delta_{h_0}+\delta_{h_0+1}+(r-3)\sum_{i=1}^{r-1}\delta_i\\
&\le & 2(\kappa-3)+(\kappa-4)+(r-3)\sum_{i=1}^{\kappa-2}\delta_i\\
&\le & 3\kappa-10+(r-3)(2\kappa-7)\\
&=& (2r-3)(\kappa-r)+2r^2-10r+11.
\end{eqnarray*}
 On the other hand, by $r+s\ge \kappa+1$, the edge number of the induced subgraph by $U_2$ is
 \begin{eqnarray*}
 e(U_2) &\geq& 1+2+\ldots +(\theta-1)+\theta(s-\theta)+1+2+\ldots+[\kappa-(r+\theta)]\\
 &=& \frac{\theta(\theta-1)}{2}+\theta s-\theta^2+\frac{(\kappa-r+1-\theta)(\kappa-r-\theta)}{2}\\
 &=&\frac{(\kappa-r)(\kappa-r+1)}{2}+(r+s-\kappa-1)\theta\\
  &\ge & \frac{(\kappa-r)(\kappa-r+1)}{2}.
    \end{eqnarray*}
 By Lemma \ref{lemma17}, we have
$$  \frac{(\kappa-r)(\kappa-r+1)}{2}\le e(U_2)\le (2r-3)(\kappa-r)+2r^2-10r+11,$$
 which implies
  \begin{equation}\label{lemma18-4}
  \kappa-r\le \frac{4r-7+\sqrt{32r^2-136r+137}}{2}.
  \end{equation}
In addition, by Lemma~\ref{lemma13} and \ref{lemma14},  we have $\delta_\kappa\le \kappa-3$ and $\delta_{\kappa-1}\le \kappa-4$.
 Therefore, by (\ref{lemma18-0}), (\ref{lemma18-2}), (\ref{lemma18-31}) and (\ref{lemma18-4}),
 \begin{eqnarray*}
 n&=& (n_0-n_t)+n_t\\
 &\le & 2\kappa-7  +\kappa+\delta_\kappa+\delta_{\kappa-1}\\
 &\le & 5\kappa-14\\
 &\le & \frac{30r-63+5\sqrt{32r^2-136r+137}}{2},
 \end{eqnarray*}
 which contradicts to the condition of $n$.
Hence the assertion  holds.
 \end{proof}

\begin{lemma}\label{lemma19}
$\alpha=\frac{1}{2}$, $m=(r-1)n-\frac{(r-1)r}{2}+3$ and $G'=\widetilde{S}_{n,m}$.
\end{lemma}

 \begin{proof}   By Lemma~\ref{lemma18}, there exists a $1\le h\le \kappa-2$ such that  $\delta_h>0$, $\delta_{h+1}=0$ and $\delta_{h+2}>0$.
 Hence by Lemma \ref{lemma15}, $\alpha=\frac{1}{2}$, $h=\kappa-2$ , $\delta_{\kappa-2}>0$, $\delta_{\kappa-1}=0$ and $\delta_\kappa=1$.
 Let $G_2$ be the connected threshold graph obtained from $G'$ by a Transformation $(\kappa+2, \kappa-1; \kappa+1, \kappa)$ of Definition \ref{def0}. By Lemma \ref{trij}, $\rho_\alpha(G_2)\ge \rho_\alpha(G')$. Note that $G_2\in \mathcal{H}_{n,m}$, which implies that $\rho_\alpha(G_2)\le \rho_\alpha(G')$. Hence  $\rho_\alpha(G_2)= \rho_\alpha(G')$, i.e., $G_2$ is  a connected threshold graph having the maximal $A_\alpha$-spectral radius of  $ \mathcal{H}_{n,m}$  with  $(r-1)n-\frac{r(r-1)}{2}<m\le rn-\frac{r(r+1)}{2}$.  Denote by the adjacency stepwise matrix $A(G_2)=(\widetilde{a}_{ij})$,  $\widetilde{\kappa}=\kappa(G_2)=\max\{j:\widetilde{a}_{j+1,j}=1 \}$ and $\widetilde{\delta}_j=\widetilde{\delta}_j(G')=|\{i: d_{G_2}(v_i)=j, i>j\}|$ for $j=1,\cdots,\widetilde{\kappa}$. Then $\widetilde{\kappa}=\kappa-1$, $\widetilde{\delta}_{\widetilde{\kappa}}=3$ by $\delta_\kappa=1$.

 Suppose that $G_2\neq S_{n,m}$.  Since $G_2$ is  a connected threshold graph having the maximal $A_\alpha$-spectral radius of  $ \mathcal{H}_{n,m}$  with  $(r-1)n-\frac{r(r-1)}{2}<m\le rn-\frac{r(r+1)}{2}$, by Lemma \ref{lemma18}, there exists a $1\le \widetilde{h}\le \widetilde{\kappa}-2$ such that  $\widetilde{\delta}_{\widetilde{h}}>0$, $\widetilde{\delta}_{\widetilde{h}+1}=0$ and $\widetilde{\delta}_{\widetilde{h}+2}>0$. Hence by
  by Lemma \ref{lemma15},  $\widetilde{\delta}_{\widetilde{\kappa}}=1$.  It is impossible.
  Therefore $G_2=S_{n,m}=K_{r-1}\vee (K_{1, \widehat{a}}\cup (n-r-\widehat{a})K_1)$. So $\widetilde{\kappa}=r$ and $\widehat{a}=\widetilde{\delta}_{\widetilde{\kappa}}=3,$  which implies  $e(G_2)=(r-1)n-\frac{r(r-1)}{2}+3$. Then
  $G'=\widetilde{S}_{n,m}$ with  $m=e(G')=(r-1)n-\frac{r(r-1)}{2}+3$.
 Hence we finish the proof of Theorem~\ref{tknc-1}.
\end{proof}

%\begin{proof} Hence we finish the proof of Theorem~\ref{tknc-1}.
%It follows from Lemma \ref{lemma19} that the assertion holds.
%\end{proof}

%\begin{rem}
%For $n+3 \le m \le 2n-3$, if $\alpha \in [\frac{1}{2},1)$, each threshold graph $G$ which is not $S_{n,m}$ contains three vertices of degree at least three. Applying the transformations in Lemmas \ref{trg} repeatedly, we can transform $G$ into $S_{n,m}$.
%\end{rem}

Now we are ready to present
{ \textbf{\mbox{Proof of Theorem \ref{tknc}}}}.
\begin{proof}
Notice that the function $f(r)=\frac{30r-63+5\sqrt{32r^2-136r+137}}{2}$ is an increasing
function with respect to $r$.  It follows from Theorem~\ref{tknc-1}  for $r\ge 3$ and $ (r-1)n-\frac{r(r-1)}{2}< m \le rn-\frac{r(r+1)}{2}$, and
Theorem~\ref{t2n-2c} for $ n-1<m \le rn-\frac{r(r+1)}{2}$ that the assertion holds.
\end{proof}

%\begin{coro}\label{t2n-1}
%Let $n\ge 5$, $m = 2n-1$, and $G'$ be the extremal graph with the maximum $Q$-spectral radius in  $\mathcal{G}_{n,m}$.\\
%{\em (1).} If $n=5$, then $G'\cong S_{n,2n-1}$.\\
%{\em (2).} If $n=6$, then $G' \cong L_{n,2n-1} $.\\
%{\em (3).} If $n=7$, then $G'\cong S_{n,2n-1}$ or $G'\cong L_{n-1,2n-1}\cup K_1$.\\
%{\em (4).} If $n=8$, then $G'\cong K_6 \cup 2K_1$.\\
%{\em (5).} If $n \ge 9$, then $G'\cong S_{n,2n-1}$.
%\end{coro}

%\begin{proof}
%The proof is similar to Corollary \ref{t2n-2}, and we omit the procedure here.
%\end{proof}

\end {document}